\documentclass[11pt]{amsart}
\usepackage{times}

\newcommand{\EE}{\mathbb{E}}

\newcommand{\GG}{\mathbb{G}}
\newcommand{\NN}{\mathbb{N}}

\newcommand{\ZZ}{\mathbb{Z}}

\newcommand{\PP}{\mathbb{P}}

\newcommand{\cE}{\mathcal{E}}

\newcommand{\cG}{\mathcal{G}}
\newcommand{\cH}{\mathcal{H}}

\newcommand{\cO}{\mathcal{O}}
\newcommand{\cP}{\mathcal{P}}

\newcommand{\cS}{\mathcal{S}}
\newcommand{\cT}{\mathcal{T}}
\newcommand{\cU}{\mathcal{U}}

\newcommand{\fA}{\mathfrak{A}}

\newcommand{\fE}{\mathfrak{E}}
\newcommand{\fe}{\mathfrak{e}}

\newcommand{\fg}{\mathfrak{g}}
\newcommand{\fh}{\mathfrak{h}}
\newcommand{\fK}{\mathfrak{K}}

\newcommand{\fM}{\mathfrak{M}}
\newcommand{\fn}{\mathfrak{n}}

\newcommand{\fsl}{\mathfrak{sl}}

\newcommand{\fu}{\mathfrak{u}}

\newcommand{\dact}{\boldsymbol{.}}
\newcommand{\lra}{\longrightarrow}

\DeclareMathOperator{\aurk}{rk_{au}}
\DeclareMathOperator{\Aut}{Aut}
\DeclareMathOperator{\Char}{char}

\DeclareMathOperator{\cx}{cx}

\DeclareMathOperator{\EIP}{EIP}
\DeclareMathOperator{\EKP}{EKP}
\DeclareMathOperator{\End}{End}

\DeclareMathOperator{\ev}{ev}
\DeclareMathOperator{\GL}{GL}
\DeclareMathOperator{\Gr}{Gr}

\DeclareMathOperator{\Hom}{Hom}
\DeclareMathOperator{\im}{im}
\DeclareMathOperator{\id}{id}

\DeclareMathOperator{\Jt}{Jt}

\DeclareMathOperator{\Lie}{Lie}
\DeclareMathOperator{\Mat}{Mat}
\DeclareMathOperator{\modd}{mod}
\DeclareMathOperator{\msim}{\mathsf{im}}
\DeclareMathOperator{\msIm}{\mathsf{Im}}

\DeclareMathOperator{\PGL}{PGL}
\DeclareMathOperator{\Pic}{Pic}
\DeclareMathOperator{\mspl}{\mathsf{pl}}
\DeclareMathOperator{\Pt}{Pt}
\DeclareMathOperator{\Pp}{P}
\DeclareMathOperator{\pr}{pr}
\DeclareMathOperator{\prk}{rk_{triv}}

\DeclareMathOperator{\Rad}{Rad}
\DeclareMathOperator{\rk}{rk}

\DeclareMathOperator{\Spec}{Spec}

\DeclareMathOperator{\Soc}{Soc}

\DeclareMathOperator{\tr}{tr}

\DeclareMathOperator{\Proj}{Proj}

\usepackage{pictex}
\usepackage{amscd}
\usepackage{latexsym}
\usepackage{amssymb}
\usepackage{eucal}
\usepackage{eufrak}
\usepackage{verbatim}

\numberwithin{equation}{section}

\newtheorem{Theorem}{Theorem}[section]
\newtheorem{Lemma}[Theorem]{Lemma}
\newtheorem{Corollary}[Theorem]{Corollary}
\newtheorem{Proposition}[Theorem]{Proposition}

\theoremstyle{Theorem}
\newtheorem{Thm}{Theorem}[subsection]
\newtheorem{Lem}[Thm]{Lemma}
\newtheorem{Prop}[Thm]{Proposition}
\newtheorem{Cor}[Thm]{Corollary}

\newtheorem*{thm*}{Theorem A}
\newtheorem*{thm**}{Theorem B}
\theoremstyle{remark}
\newtheorem*{Remark}{Remark}
\newtheorem*{Remarks}{Remarks}
\newtheorem*{Definition}{Definition}
\newtheorem*{Example}{Example}
\newtheorem*{Examples}{Examples}

\numberwithin{equation}{section}
\begin{document}

\title{Representations of Finite Group Schemes and Morphisms of Projective Varieties}

\author{Rolf Farnsteiner}

\address[R. Farnsteiner]{Mathematisches Seminar, Christian-Albrechts-Universit\"at zu Kiel, Ludewig-Meyn-Str. 4, 24098 Kiel, Germany}
\email{rolf@math.uni-kiel.de}
\thanks{Supported by the D.F.G. priority program SPP1388 `Darstellungstheorie'.}
\date{\today}

\makeatletter
%\@addtoreset{subabschnitt}{abschnitt}
\makeatother


\begin{abstract} Given a finite group scheme $\cG$ over an algebraically closed field $k$ of characteristic $\Char(k)=p>0$, we introduce new invariants for a $\cG$-module $M$ by associating 
certain morphisms $\deg^j_M : U_M \lra \Gr_d(M) \ \ (1\!\le\!j\!\le\! p\!-\!1)$ to $M$ that take values in Grassmannians of $M$. These maps are studied for two classes of finite algebraic groups, 
infinitesimal group schemes and elementary abelian group schemes. The maps associated to the so-called modules of constant $j$-rank have a well-defined degree ranging between $0$ and 
$j\rk^j(M)$, where $\rk^j(M)$ is the generic $j$-rank  of $M$. The extreme values are attained when the module $M$ has the equal images property or the equal kernels property. We establish 
a formula linking the $j$-degrees of $M$ and its dual $M^\ast$. For a self-dual module $M$ of constant Jordan type this provides information concerning the indecomposable constituents of the 
pull-back $\alpha^\ast(M)$ of $M$ along a $p$-point $\alpha : k[X]/(X^p) \lra k\cG$.  \end{abstract}

\maketitle

\section*{Introduction}
Let $\cG$ be a finite group scheme over an algebraically closed field $k$ of characteristic $p>0$. Much of the recent work on the representations of $\cG$ has
focused on the investigation of invariants that are defined in terms of representation-theoretic support spaces, whose elements are equivalence classes of certain algebra homomorphisms
$\alpha : k[T]/(T^p) \lra k\cG$, the so-called $p$-points. For each $\cG$-module $M$, one can consider the linear operators $m\mapsto \alpha(T\!+\!(T^p))m$ along with their images, kernels and ranks. 
By specifying values of these data, one arrives at interesting full subcategories of the category $\modd \cG$ of finite-dimensional $\cG$-modules. In this article, we show how morphisms with values in 
Grassmannians can be employed to obtain new invariants for the objects of these categories. 

One salient feature of the modular representation theory of finite groups is given by reduction to elementary abelian groups, with Quillen's Dimension Theorem being one notable instance. In
our situation, basic algebro-geometric observations imply that our invariants are determined by their values on elementary abelian group schemes of rank $2$. While these groups usually still have wild 
representation type, their modules enjoy properties that do not possess analogs in higher ranks. 

Following a few preliminary observations concerning morphisms between projective varieties, we turn in Section \ref{S:MIG} to the study of maps defined by modules. To a $\cG$-module
$M$, one associates its generic $j$-rank $\rk^j(M)$ ($1\!\le\!j\!\le\! p\!-\!1$), which is the maximal rank associated to the $j$-th powers of the aforementioned operators. If $\cG$ is an infinitesimal group scheme, a $\cG$-module $M$ thus determines open subsets $U_{M,j} \subseteq \Proj(V(\cG))$ of the projectivized variety of infinitesimal 
one-parameter subgroups. Moreover, the images of these operators give rise to morphisms  $\msim^j_M : U_{M,j} \lra \Gr_{\rk^j(M)}(M)$ taking values in Grassmannians with base space $M$. 
For the so-called modules of constant $j$-rank, which are characterized by the condition $U_{M,j}=\Proj(V(\cG))$, one particular case of interest arises when this variety coincides with a projective 
space $\PP^n$. In the context of 
restricted Lie algebras (or infinitesimal groups of height $1$), this happens for instance when the nullcone $V(\fg) \subseteq \fg$ is a linear subspace. In that case, the resulting morphisms $\msim^j_M: 
\PP(V(\fg)) \lra \Gr_{\rk^j(M)}(M)$ are constant or have finite generic fibres. For $j=1$, the map $\msim_M^j$ is constant or injective, so that not all morphisms arise via this construction. Modules yielding 
constant morphisms enjoy the so-called equal images property. They were first investigated for the group $\ZZ/(p)\!\times\!\ZZ/(p)$ by Carlson-Friedlander-Suslin in \cite{CFS}.

Upon composing $\msim^j_M$ with the appropriate Pl\"ucker embedding, we arrive at morphisms between projective spaces. The common degree of the homogeneous polynomials defining these 
maps is an interesting invariant, the $j$-degree $\deg^j(M)$ of $M$, which we study in Section \ref{S:DM}. The degrees and ranks are linked via the following formula:

\begin{thm*} Let $\fg$ be a restricted Lie algebra whose nullcone $V(\fg)$ is a subspace of $\fg$. If $M$ is a $\fg$-module of constant $j$-rank, then
\[ \deg^j(M)\!+\!\deg^j(M^\ast)=j\rk^j(M).\]
\end{thm*}
Consequently, the categories of equal images modules and their duals, the equal kernels modules, which were defined in \cite{CFS}, comprise exactly those modules, whose $j$-degrees are $0$ and $j\rk^j(M)$, respectively. 
Moreover, degrees may be used to distinguish modules having the same constant Jordan type. The functions $M \mapsto \deg^j(M)$ are subadditive on exact sequences of modules of constant 
$j$-rank, and additive on sequences that are locally split. In the special case of elementary abelian group schemes of rank $2$, the $1$-degree of a module $M$ of constant $1$-rank coincides with the 
codimension of its ``generic kernel'' $\fK(M)\subseteq M$.

Although analogs of the above maps don't seem to be available for arbitrary finite group schemes, the factorization property of $p$-points often allows the extension of results concerning restricted
Lie algebras to the general context. A case in point is provided in Sections \ref{S:CRLD} and \ref{S:FG}, where we exploit results by Tango \cite{Ta1,Ta2} on morphisms $\PP^n \lra \Gr_d(V)$ to 
obtain information on modules of constant rank. As we show, a constant rank module over a finite group $G$ has generic rank zero whenever its dimension is bounded by the $p$-rank $\rk_p(G)$ 
of $G$. Thus, non-trivial modules of constant rank for a $p$-elementary abelian group of rank $r$ have dimension $\ge r\!+\!1$.  In the same vein, the following result, which yields information on the 
Jordan types of self-dual modules, rests on the observation that the degree of a module is computable from its restriction to an elementary abelian subgroup scheme of rank $2$.

\begin{thm**} Let $\cG$ be a finite group scheme containing an elementary abelian subgroup scheme of rank $\ge 2$. Suppose that $M$ is a self-dual $\cG$-module.
\begin{enumerate}
\item If $M$ has constant $j$-rank, then $\rk^j(M)\equiv 0 \modd(2)$, whenever $j \equiv 1\modd(2)$.
\item If $M$ has constant Jordan type $\Jt(M)=\bigoplus_{i=1}^pa_i[i]$, then $a_i\equiv 0 \ \modd(2)$ whenever $i\equiv 0 \ \modd(2)$. \end{enumerate} \end{thm**}
The number $a_i$ above is the multiplicity of the $i$-dimensional indecomposable $k[X]/(X^p)$-module $[i]$ as a direct summand of the module $\alpha^\ast(M)$, obtained from $M$ via
pull-back along the $p$-point $\alpha : k[X]/(X^p)\lra k\cG$.
 
This paper mainly follows an algebraic approach which seems to be suitable for our purposes. Geometric aspects, related to alternative methods involving vector bundles, are only alluded to 
occasionally. I am grateful to Eric Friedlander and Julia Pevtsova for sharing their geometric insights with me. 

\bigskip

\section{Morphisms and homogeneous polynomials}\label{S:MHP}
In this section we collect a few basic properties of certain morphisms that are relevant for our intended applications. Our main tool is the notion of a degree of a morphism
$\varphi : X \lra Y$ between certain quasi-projective varieties. 

\bigskip

\subsection{Morphisms between projective varieties} 
Throughout this section, $k$ denotes an algebraically closed field. Recall that a polynomial $f \in k[X_0,\ldots, X_n]$ is referred to as being {\it homogeneous of degree $d$} if $f$ is a linear combination 
of monomials of degree $d$. We let $k[X_0,\ldots,X_n]_d$ be the subspace of homogeneous polynomials of degree $d$ and put $\deg(f)=d$ for every $f \in k[X_0,\ldots,X_n]_d\!\smallsetminus\!\{0\}$.

Given $f,g \in k[X_0,\ldots,X_m]$, we write $g|f$ to indicate that $g$ divides $f$.

\bigskip

\begin{Lem} \label{MPS1} Let $f\ne 0$ be homogeneous of degree $d$. If $g |f$, then $g$ is homogeneous of degree $\le d$. \end{Lem}

\begin{proof} We write $f=gh$ as well as $g=\sum_{i=\ell_1}^{\ell_2} g_i$, $h =\sum_{i=m_1}^{m_2} h_i$, where $g_i,h_i$ are homogeneous of degree $i$ and $g_{\ell_1},g_{\ell_2},h_{m_1},h_{m_2}$
are not zero. Then we have
\[ f = g_{\ell_2}h_{m_2}\!+\! \sum_{i<\ell_2+m_2}v_i,\]
where $v_i \in k[X_0,\ldots,X_n]_i$. Since $g_{\ell_2}h_{m_2} \ne 0$, this readily yields $d=\ell_2\!+\!m_2$. By the same token,
\[  f = g_{\ell_1}h_{m_1}\!+\! \sum_{i>\ell_1+m_1}w_i \ \ ;  \ \ \deg(w_i) = i,\]
so that $d=\ell_1\!+\!m_1$. Hence $\ell_2 = d\!-\!m_2 \le d\!-\!m_1 = \ell_1$, implying that $g=g_{\ell_1}$ is homogeneous of degree $\le d$. \end{proof}

\bigskip
\noindent
For an $n$-dimensional $k$-vector space $V\ne (0)$, we let $\PP(V)$ be the projective variety of one-dimensional subspaces of $V$. In particular, $\PP^n = \PP(k^{n+1})$ denotes the
$n$-dimensional projective space, whose elements are of the form $(x_0\!:\!x_1\!: \cdots :\!x_n)$.

We are interested in morphisms $\varphi: X \lra Y$ between quasi-projective varieties that are given by homogeneous polynomials in the sense of the following:

\bigskip

\begin{Definition} Let $X \subseteq \PP^n$ and $Y \subseteq \PP^m$ be quasi-projective varieties, $\varphi : X \lra Y$ be a morphism, $f_0,\ldots, f_m \in k[X_0,\ldots,X_n]$ be homogeneous 
polynomials of degree $d$. We say that \textit{$\varphi$ is defined by $(f_0,\ldots, f_m)$} and that $(f_0,\ldots, f_m)$ is a {\it defining system for} $\varphi$, provided
\[ \varphi(x) = (f_0(x):\cdots :f_m(x)) \ \ \ \ \text{for every} \ x \in X.\] 
The morphism $\varphi$ is said to be \textit{homogeneous}, if it is defined by some tuple of homogeneous polynomials of the same degree. \end{Definition}
\noindent
In this definition we tacitly assume that $X$ does not intersect the zero locus $Z(f_0,\ldots,f_m)$ of the polynomials $f_0,\ldots,f_m$.  

We begin with a few elementary observations concerning rational maps $\varphi : \PP^n \dashrightarrow \PP^m$. If $f \in k[X_0,\ldots,X_n]$ is homogeneous, then $D(f) := \{x \in \PP^n \ ; \ f(x) \ne 0\}$ is 
a well-defined open subset of $\PP^n$. 

\bigskip

\begin{Lem} \label{MPS2} Let $U \subseteq \PP^n$ be a non-empty open subset, $\varphi : U \lra \PP^m$ be a morphism. 
\begin{enumerate}
\item Suppose there are polynomials $f_0,\ldots,f_m \in k[X_0,\ldots,X_n]_d$ and $g_0,\ldots,g_m \in k[X_0,\ldots,X_n]_{d'}$ such that
\begin{enumerate}
\item $\gcd(f_0,\ldots,f_m) = 1 = \gcd(g_0,\ldots,g_m)$, and
\item $\varphi$ is defined by $(f_0,\ldots,f_m)$ as well as $(g_0,\ldots,g_m)$. \end{enumerate} 
Then there exists $\lambda \in k^\times$ such that $g_i = \lambda f_i$ for $i \in \{0,\ldots,m\}$.
\item If $\varphi$ is defined by $(g_0,\ldots,g_m)$, then there exist homogeneous polynomials $h$ and $f_0,\ldots, f_m \in k[X_0,\ldots,X_n]$ such that
\begin{enumerate}
\item $U \subseteq D(h)$, $g_i = hf_i$ for $0\le i\le m$, and
\item $\gcd(f_0,\ldots,f_m)=1$, and
\item $\varphi$ is defined by $(f_0,\ldots,f_m)$. \end{enumerate} \end{enumerate} \end{Lem}

\begin{proof} (1) By assumption (b), we have $f_i(u)g_j(u)=f_j(u)g_i(u)$ for all $u \in U$ and $i,j \in \{0,\ldots,m\}$. As the open set $U$ lies dense in the irreducible variety $\PP^n$, this readily yields
\[ (\ast) \ \ \ \ \ \ \ f_ig_j = f_jg_i \ \ \ \ \ \ \ \ \ \ \forall \ i,j \in \{0,\ldots,m\}.\]
Let $i \in \{0,\ldots,m\}$. If $f_i=0$, then ($\ast$) yields $f_jg_i=0$ for all $j$. Since there is $j \in \{0,\ldots,m\}$ such that $f_j\ne 0$, it follows that $g_i=0$.

Suppose that $f_i\ne 0$, so that $g_i\ne 0$. For a prime polynomial $p \in k[X_0,\ldots,X_n]$, we denote by $m_p(f_i)$ the multiplicity of $p$ in $f_i$. Suppose that $m_p(f_i)>m_p(g_i)$. Then ($\ast$) 
implies
\[ m_p(f_j)=m_p(f_jg_i)\!-\!m_p(g_i) = m_p(f_ig_j)\!-\!m_p(g_i) = m_p(g_j)\!+\!m_p(f_i)\!-\!m_p(g_i)> m_p(g_j)\]
for every $j \in \{0,\ldots,m\}$ with $f_j\ne 0$, so that $p$ is a common divisor of the $f_j$, a contradiction. Thus, $m_p(f_i)\le m_p(g_i)$, whence $m_p(f_i)=m_p(g_i)$ by symmetry. As a result, the 
polynomial $g_i$ is a scalar multiple of $f_i$. The assertion now follows from ($\ast$).

(2) Let $h$ be a greatest common divisor of $g_0,\ldots,g_m$, so that $g_i=hf_i$ for some $f_i \in k[X_0,\ldots,X_n]$.  By virtue of Lemma \ref{MPS1}, the polynomials $h$ and $f_i$ are homogeneous. 
Hence the $f_i$ are homogeneous of the same degree and have greatest common divisor $1$. 

Given $u=[x] \in U$, there is $i \in \{0,\ldots,m\}$ such that 
\[ 0 \ne g_i(x) = h(x)f_i(x),\] 
so that $U\subseteq D(h)$. Since
\[ \varphi(u) = (g_0(u):\cdots :g_m(u)) = (f_0(u):\cdots :f_m(u)) \ \ \ \ \text{for all} \ u \in U, \]
the map $\varphi$ is defined by $(f_0,\ldots,f_m)$. \end{proof}

\bigskip

\begin{Definition} Let $X \subseteq \PP^n$ and $Y \subseteq \PP^m$ be quasi-projective varieties, $\varphi : X \lra Y$ be a homogeneous morphism. An $(m\!+\!1)$-tuple $(f_0,\ldots, f_m) \in k
[X_0,\ldots,X_n]_d^{m+1}$ is called a {\it reduced defining system for $\varphi$}, provided
\begin{enumerate}
\item $\varphi$ is defined by $(f_0,\ldots,f_m)$, and
\item $\gcd(f_0,\ldots,f_m)=1$. \end{enumerate}\end{Definition}

\bigskip

\begin{Lem} \label{MPS3} Let $U \subseteq \PP^n$ be a non-empty open subset, $\varphi : U \lra \PP^m$ be a morphism. Then $\varphi$ is homogeneous. \end{Lem}

\begin{proof} (1) Since $U \subseteq \PP^n$ is a noetherian topological space, an application of \cite[(1.65)]{GW} provides non-empty open subsets $U_1,\ldots, U_r$ of $U$, and homogeneous 
polynomials $f_{ij} \in k[X_0,\ldots,X_n]$, where $0\le i \le m$ and $1\le j \le r$, such that
\begin{enumerate}
\item[(a)] $U = \bigcup_{j=1}^rU_j$,
\item[(b)] $U_j \subseteq \bigcup_{i=0}^m D(f_{ij})$,
\item[(c)] there exist $d_1, \ldots, d_r  \in \NN_0$ such that $f_{ij} \in k[X_0,\ldots,X_n]_{d_j}$ for $0\le i \le m$ and $1\le j\le r$,
\item[(d)] $\varphi(u) = (f_{0j}(u):\cdots:f_{mj}(u))$ for all $u \in U_j$ and $j \in \{1,\ldots, r\}$. \end{enumerate}
Since $U_j \subseteq \PP^n$ is open, Lemma \ref{MPS2} shows that we may assume $\gcd(f_{0j},\ldots, f_{mj})=1$ for every $j \in \{1,\ldots, r\}$. Given $j \in \{1,\ldots,r\}$, we 
see that $U_j\cap U_1$ is a non-empty open subset of $\PP^n$. By applying Lemma \ref{MPS2}(1) to the morphism $\varphi|_{U_j\cap U_1}$ we find elements $\lambda_j \in k^\times$ such that 
$f_{ij} = \lambda_jf_{i1}$ for $0\!\le\! i\! \le\! m$. Setting $f_i := f_{i1}$ for $0\!\le\! i\!\le\! m$, we thus obtain for $u \in U_j$
\[ \varphi(u) = (f_{0j}(u):\cdots: f_{mj}(u)) = (f_0(u):\cdots:f_m(u)).\]
Thanks to property (a), this identity holds for all $u \in U$. Consequently, the map $\varphi$ is a homogeneous morphism. \end{proof}

\bigskip
\noindent
In view of the foregoing result, the following definition is meaningful:

\bigskip

\begin{Definition} Let $\varphi : \PP^n \dashrightarrow \PP^m$ be a rational map, $(f_0,\ldots,f_m)$ be a reduced defining system for $\varphi$. Then 
\[ \deg(\varphi) = \deg(f_i) \ \ ; \ \ f_i \ne 0\]
is called the {\it degree} $\deg(\varphi)$ of $\varphi$. \end{Definition}

\bigskip

\begin{Remarks} (1) \ The above definition should not be confused with the projective degree of a morphism, see \cite[(19.4)]{Har}. If $\nu_d :\PP^n \lra \PP^N$ denotes the $d$-th Veronese embedding, 
then $\deg(\nu_d) = d$, while the degree of the image $\im \nu_d$ is $d^n$ (cf.\ \cite[(18.13)]{Har}), so that the map $\nu_d : \PP^n \lra \im \nu_d$ has projective degree $d^n$. 

(2) \ Let $A=(a_{ij}) \in \GL_{m+1}(k)$. If $(f_0,\ldots, f_m) \in k[X_0,\ldots,X_m]_d^{m+1}$, then $g_j := \sum_{i=0}^ma_{ij}f_i$ belongs to $k[X_0,\ldots,X_n]_d$, and $\gcd(g_0,\ldots, g_m)=
\gcd(f_0,\ldots,f_m)$. Accordingly, a linear change of coordinates does not affect the degree of a rational morphism $\PP^n \dashrightarrow \PP^m$. 

(3) \ A morphism $\varphi : \PP^n \lra \PP^m$ defines a homomorphism $\varphi^\ast : \Pic(\PP^m) \lra \Pic(\PP^n)$ between the Picard groups of line bundles. 
Since these groups are isomorphic to $\ZZ$, the degree $d$ of $\varphi$ is given by $\varphi^\ast(\cO_{\PP^m}(1))=\cO_{\PP^n}(d)$. \end{Remarks}

\bigskip

\begin{Cor} \label{MPS4} If $\varphi : \PP^n \lra \PP^m$ is a morphism, then any defining system is a reduced defining system. \end{Cor}

\begin{proof} Let $(g_0,\ldots,g_m)$ be a defining system, $(f_0,\ldots,f_m)$ be a reduced defining system for $\varphi$. Owing to Lemma \ref{MPS2}(2), there exists a homogeneous
polynomial $h$ such that $\PP^n = D(h)$ and $g_i = hf_i$ for $0\le i\le m$. Since such a polynomial is constant, our assertion follows. \end{proof} 

\bigskip

\begin{Cor} \label{MPS5} Let $\varphi : \PP^n \lra \PP^m$ and $\psi : \PP^m \lra \PP^s$ be morphisms.
\begin{enumerate}
\item We have $\deg(\psi\!\circ\!\varphi) = \deg(\psi)\deg(\varphi)$.
\item If $\psi\!\circ\! \varphi$ is constant, then $\psi$ is constant or $\varphi$ is constant. \end{enumerate} \end{Cor}

\begin{proof} (1) If $f_0,\ldots, f_m \in k[X_0,\ldots,X_n]$ are homogeneous polynomials of degree $d$ and $\nu \in \NN_0^{m+1}$, then $f_0^{\nu_0}\cdots f_m^{\nu_m}$ is homogeneous of degree 
$d\cdot (\sum_{i=0}^m\nu_i)$. Consequently, $\psi\!\circ\! \varphi$ affords a defining system of homogeneous polynomials of degree $\deg(\psi)\deg(\varphi)$ and Corollary \ref{MPS4} yields the
result.

(2) Since $\psi\!\circ\!\varphi$ is constant, there is $c=(c_0,\ldots,c_s) \in k^{s+1}\!\smallsetminus\!\{0\}$ such that $[\psi\circ \varphi](x) = (c_0:\cdots :c_s)$ for all $x 
\in \PP^n$. This readily implies $0=\deg(\psi\!\circ\!\varphi) = \deg(\psi)\deg(\varphi)$. As a result, one of the factors has degree $0$ and the corresponding map is constant. \end{proof}

\bigskip

\begin{Remarks} (1) In view of the observations above, a morphism $\varphi : \PP^n \lra \PP^m$ corresponds to an element $[A] \in \PP(\Mat_{(m+1)\times(n+1)}(k))$ that is defined by a matrix
$A$ of rank $n\!+\!1$.

(2) According to Corollary \ref{MPS5}, automorphisms of $\PP^n$ have degree $1$. Hence the canonical action of $\GL_{n+1}(k)$ on $\PP^n$ induces an isomorphism  $\PGL_{n+1}(k)\cong 
\Aut(\PP^n)$. \end{Remarks}

\bigskip

\begin{Cor} \label{MPS6} Let $X \subseteq \PP^n$ be a quasi-projective variety such that there exists a non-constant morphism $\omega_X : \PP^1 \lra \PP^n$ with $\omega_X(\PP^1)\subseteq X$. 
If $\varphi : X \lra \PP^m$ is a homogeneous morphism, then 
\[ \deg(\varphi) := \frac{\deg(\varphi\!\circ\!\omega_X)}{\deg(\omega_X)} \in \NN_0\] 
is independent of the choice of $\omega_X$. \end{Cor}

\begin{proof}  Lemma \ref{MPS3} provides homogeneous polynomials $\omega_0,\ldots, \omega_n \in k[X,Y]$ of degree $\ell>0$ such that
\[ \omega_X(x) = (\omega_0(x):\cdots : \omega_n(x)) \ \ \ \ \ \ \forall \ x \in \PP^1.\]
As $\varphi$ is homogeneous, we can find homogeneous polynomials $f_0,\ldots, f_m$ of degree $d$, say, such that
\[ \varphi(x) = (f_0(x):\cdots : f_m(x)) \ \ \ \ \ \ \forall \ x \in X.\]
Note that the polynomials $g_i := f_i(\omega_0,\ldots,\omega_n) \in k[X,Y]$ are homogeneous of degree $d\ell$ such that
\[ (\varphi\circ \omega_X)(x) = (g_0(x):\cdots : g_m(x)) \ \ \ \ \forall \ x \in \PP^1.\]
According to Corollary \ref{MPS4}, this implies that $\deg(\varphi\circ \omega_X) = d\ell$, so that $\deg(\varphi)=d \in \NN_0$ does not depend on $\omega_X$.  \end{proof} 

\bigskip

\subsection{Tango's Theorem} It is a well-known fact that the dimension of the image of a morphism $\varphi : \PP^m \lra \PP^n$ belongs to $\{0,m\}$, cf.\ \cite[I.7,Prop.6]{Mu}. 
In a series of articles, including \cite{Ta1,Ta2}, H.\ Tango investigated morphisms $\PP^m \lra \Gr(n,d)$ with values in the Grassmann variety $\Gr(n,d)$ of $d$-dimensional linear 
subspaces of $\PP^n$. Letting $\Gr_d(V)$ be the Grassmann variety of $d$-dimensional subspaces of a finite-dimensional $k$-vector space $V$, we summarize some of his results 
as follows:

\bigskip

\begin{Thm}[H. Tango] \label{Ta1} Let $V$ be an $n$-dimensional vector space, $\varphi : \PP^m \lra \Gr_d(V)$ be a morphism.
\begin{enumerate}
\item If $n\le m$, then $\varphi$ is constant.
\item Suppose that $n=m\!+\!1$.
\begin{enumerate}
\item If $n$ is odd and $2\! \le\! d\! \le \!n\!-\!2$, then $\varphi$ is constant.
\item If $d$ is odd while $3\! \le\! d \!\le\! n\!-\!2$ and $(n,d)\ne (6,3)$, then $\varphi$ is constant. \end{enumerate} \end{enumerate} \end{Thm}

\begin{proof} By definition, we have $\Gr_d(V) \cong \Gr(n\!-\!1,d\!-\!1)$. Now (1) and (2) are direct consequences of \cite[(3.2)]{Ta1} and \cite[Thm.]{Ta2}, respectively. \end{proof} 

\bigskip

\subsection{Homogeneous morphisms of conical varieties}
Let $\fM$ be a finite-dimensional $k^\times$-module with weight space decomposition $\fM=\bigoplus_{i\ge 0}\fM_i$. A $k^\times$-stable Zariski closed
subset $V\subseteq \fM$ such that $V\not\subseteq \fM_0$ is referred to as a {\it conical variety}. The commutative group $k^\times$ acts on the coordinate ring $k[\fM] = S(\fM^\ast)$ of polynomial
functions on $\fM$ via
\[ (\alpha\dact f)(m) := f(\alpha.m) \ \ \ \ \ \ \ \ \ \forall \ \alpha \in k^\times, \, m \in \fM, \, f \in k[\fM]\]
such that each component $k[\fM]^{(d)} := S^d(\fM^\ast)$ is a $k^\times$-submodule. Since $V$ is conical, its coordinate ring $k[V]$ inherits this action and there results a grading
\[ k[V] = \bigoplus_{i\ge 0} k[V]_i,\]
where $k[V]_i = \{ f \in k[V] \ ; \ f(\alpha.v)=\alpha^if(v) \ \ \ \forall \ \alpha \in k^\times, \, v \in V\}$. The elements of $k[V]_i$ will be referred to as homogeneous polynomial functions of degree $i$. 

In the above situation, we consider the projective varieties $\Proj(V)\subseteq \Proj(\fM)$.  The underlying sets are the $k^\times$-orbits of $V\smallsetminus \fM_0\subseteq \fM\smallsetminus\fM_0$.
The regular functions are locally given by fractions of homogeneous polynomial functions of the same degree. If the action of $k^\times$ on $\fM$ is just the restriction of the scalar multiplication,
so that $\fM=\fM_1$, we retrieve the standard projective varieties $\PP(V)\subseteq \PP(\fM)$. We write $V_0 := \{v \in V \ ; \ \alpha.v = v \ \ \ \forall \ \alpha \in k^\times\} = V\cap \fM_0$.

A morphism $\varphi : V \lra W$ between two conical varieties is called {\it homogeneous}, provided there exists $d\!\ge\!0$ such that $\varphi(\alpha.v)=\alpha^d.\varphi(v)$ for all $\alpha \in k^\times$
and $v \in V$. This requirement is equivalent to the comorphism $\varphi^\ast : k[W] \lra k[V]$ satisfying $\varphi^\ast(k[W]_n) \subseteq k[V]_{nd}$ for all $n\ge 0$.

\bigskip

\begin{Lem} \label{HM1} Let $\varphi : V \lra W$ be a homogeneous morphism of conical affine varieties. Then the following statements hold:
\begin{enumerate}
\item $\cO_\varphi := \{[v] \in \Proj(V) \ ; \ \varphi(v)\not \in  W_0\}$ is an open subset of $\Proj(V)$.
\item The map $\bar{\varphi} : \cO_\varphi \lra \Proj(W) \ \ ; \ \ [v] \mapsto [\varphi(v)]$ is a morphism of varieties.
\end{enumerate} \end{Lem}

\begin{proof} (1) Suppose that $\varphi : V \lra W$ has degree $d$. If $W\subseteq \fM$, we choose a basis $\{v_1,\ldots,v_m\}$ of $\fM$ consisting of homogeneous vectors such that 
$\{v_1,\ldots,v_n\}$
is a basis of  $\bigoplus_{i>0}\fM_i$. For each $j \in \{1,\ldots,n\}$, the coordinate function ${\rm pr}_j: \fM \lra k$ defines a homogeneous element of $k[W]$ of degree $\deg(v_j)$. Consequently, $
\varphi_j := {\rm pr}_j\circ \varphi \in k[V]$ is homogeneous for $1\le j \le n$ and $\cO_\varphi = \Proj(V)\!\smallsetminus\!Z(\varphi_1,\ldots,\varphi_n)$, the complement of the zero locus of the 
$\varphi_j$, is an open subset of $\Proj(V)$.

(2) By assumption, we have $\varphi(\alpha.v)=\alpha^d.\varphi(v)$ for all $v \in V$, so that the map $\bar{\varphi}$ is well-defined and continuous. Let $U\subseteq \Proj(W)$ be an open set,
$\rho : U \lra k$ be a regular function. Then $U':=\bar{\varphi}^{-1}(U)$ is open in $\cO_\varphi$. Given $x \in U'$, there exist an open subset $U_1 \subseteq U$ containing $\bar{\varphi}(x)$ and
$f,g \in k[W]$ homogeneous of the same degree with $U_1\subseteq D(g)$ and such that $\rho(y)= \frac{f(y)}{g(y)}$ for all $y\in U_1$. Since $\varphi$ is homogeneous, the functions $f\circ \varphi,
g\circ\varphi \in k[V]$ are homogeneous of the same degree, and for $[u] \in \bar{\varphi}^{-1}(U_1)$, we have
\[ \rho(\bar{\varphi}([u])) = \frac{f(\varphi(u))}{g(\varphi(u))},\]
so that $\rho\circ\bar{\varphi}: U' \lra k$ is regular at $x$. As a result, the map $\bar{\varphi} : \cO_\varphi \lra \Proj(W)$ is a morphism. \end{proof}

\bigskip

\section{Modules for Infinitesimal Group Schemes and Maps to Grassmannians} \label{S:MIG}
Let $\cG$ be a infinitesimal group scheme. In this section we associate to every $\cG$-module $M$ several morphisms $\msim^j_M$ which take values in Grassmannians and are defined on open
subsets of the projectivized rank variety $\Proj(V(\cG))$ of infinitesimal one-parameter subgroups of $\cG$.  

\bigskip

\subsection{Preliminaries} As before, $k$ denotes an algebraically closed field. Let $V$ be an $n$-dimensional $k$-vector space with basis $\{v_1,\ldots,v_n\}$. Given $d \in \{1,\ldots,n\}$, we denote by $\cS(d)$ the set of $d$-element subsets of $\{1,\ldots,n\}$. For $J \in \cS(d)$
we put $V_J := \bigoplus_{j\in J}kv_j$ and define
\[ v_J := v_{j_1}\wedge v_{j_2}\wedge \cdots \wedge v_{j_d} \in \bigwedge^d(V),\]
where $j_1<j_2<\cdots<j_d$ belong to $J$. For an endomorphism $f : V \lra V$, we denote by $\bigwedge^d(f)$ the unique endomorphism of $\bigwedge^d(V)$
such that $\bigwedge^d(f)(w_1\wedge w_2\wedge \cdots \wedge w_d) = f(w_1)\wedge f(w_2)\wedge \cdots \wedge f(w_d)$ for all $w_1,\ldots, w_d \in V$.
There results a homogeneous morphism $\End_k(V) \lra \End_k(\bigwedge^d(V)) \ ; \ f \mapsto \bigwedge^d(f)$ of degree $d$.

\bigskip

\begin{Lem} \label{Pr1} Given $d \in \{1,\ldots,n\}$ and $J \in \cS(d)$, the following statements hold:
\begin{enumerate}
\item The set $\cO_J := \{[f] \in \PP(\End_k(V)) \ ; \ \bigwedge^d(f)(v_J) \ne 0\}$ is an open subset of $\PP(\End_k(V))$.
\item The map $\overline{\rm ev}_J : \cO_J \lra \PP(\bigwedge^d(V)) \ ; \ [f] \mapsto [\bigwedge^d(f)(v_J)]$ is a morphism of varieties.
\end{enumerate} \end{Lem}

\begin{proof} Being the composite of homogeneous morphisms of degrees $d$ and $1$, the map
\[ {\rm ev}_J : \End_k(V) \lra \bigwedge^d(V) \ \ ; \ \ f \mapsto \bigwedge^d(f)(v_J)\]
is a homogeneous morphism of conical varieties of degree $d$. Since $\cO_J=\cO_{{\rm ev}_J}$, our assertions follow from Lemma \ref{HM1}. \end{proof}

\bigskip
\noindent
Note that $\bigcup_{J\in \cS(d)} \cO_J = \{ [f] \in \PP(\End_k(V)) \ ; \ \rk(f)\ge d\}$. Consequently, the subset
\[ \PP(\End_k(V))_d := \{[f] \in \PP(\End_k(V)) \ ; \ \rk(f)=d\}\]
of $\PP(\End_k(V))$ is locally closed. Thus, $\PP(\End_k(V))_d$ is a quasi-projective variety, which, being an orbit under the canonical action of $\GL(V)\!\times\!\GL(V)$, is irreducible. Setting $\cU_J := \cO_J\cap\PP(\End_k(V))_d$, we obtain an open covering $\PP(\End_k(V))_d=\bigcup_{J \in \cS(d)} \cU_J$. 

We let
\[ \mspl_V : \Gr_d(V) \lra \PP(\bigwedge^d(V)) \ \ ; \ \ W \mapsto \bigwedge^d(W)\]
be the {\it Pl\"ucker embedding}, which identifies $\Gr_d(V)$ with the closed subset $\im \mspl_V \subseteq \PP(\bigwedge^d(V))$.

According to Lemma \ref{Pr1}, the map
\[ \overline{\rm ev}_J : \cU_J \lra \PP(\bigwedge^d(V)) \ \ ; \ \ [f] \mapsto [\bigwedge^d(f)(v_J)]\]
is a morphism of varieties for every $J \in \cS(d)$. These maps can be glued to a morphism $\PP(\End_k(V))_d \lra \Gr_d(V)$.

\bigskip

\begin{Prop} \label{Pr2} Let $d \in \{1,\ldots,n\}$. Then the following statements hold:
\begin{enumerate}
\item The map
\[\overline{\rm ev} : \PP(\End_k(V))_d \lra \PP(\bigwedge^d(V)) \ \  ; \ \  [f] \mapsto \bigwedge^d(f(V))\]
is a morphism of varieties.
\item The map
\[\overline{\im} : \PP(\End_k(V))_d \lra \Gr_d(V) \ \ ; \ \ [f] \mapsto f(V)\]
is a morphism of varieties. \end{enumerate} \end{Prop}

\begin{proof} (1) Suppose that $f \in \cU_J\cap\cU_{J'}$ for two subsets $J,J' \in \cS(d)$. Then $\bigwedge^d(f)(v_J),\bigwedge^d(f)(v_{J'})$ are nonzero elements of $\bigwedge^d(V)$ belonging to
the subspace $\bigwedge^d(f(V)) \subseteq \bigwedge^d(V)$. Since $\rk(f)= d$, the space $\bigwedge^d(f(V))$ is one-dimensional, so that
\[ \overline{\rm ev}_J([f]) = [\bigwedge^d(f)(v_J)] = [\bigwedge^d(f)(v_{J'})] = \overline{\rm ev}_{J'}([f]).\]
Owing to Lemma \ref{Pr1}, the unique map $\overline{\rm ev} : \PP(\End_k(V))_d \lra \PP(\bigwedge^d(V))$, given by
\[ \overline{\rm ev}|_{\cU_J} = \overline{\rm ev}_J\]
for all $J \in \cS(d)$ is a morphism of varieties.

(2) Let $J \in \cS(d)$ be a $d$-element subset. Given $[f] \in \cU_J$, we have $\bigwedge^d(f)(v_J)\ne 0$. Consequently, $\dim_kf(V_J) = d = \dim_k f(V)$, so that
\[ (\mspl_V \circ\overline{\im}) ([f]) = \bigwedge^d(f(V)) = \bigwedge^d(f(V_J)) = [\bigwedge^d(f)(v_J)] = \overline{\rm ev}([f]).\]
As a result, $\overline{\rm ev} = \mspl_V\circ\overline{\im}$. Since $\overline{\rm ev}$ is a morphism, so is $\overline{\im}$. \end{proof}

\bigskip
\noindent 
Let $f \in \End_k(V)$ and denote by $A(f)=(a_{ij}) \in {\rm Mat}_n(k)$ the ($n\!\times\!n$)-matrix representing $f$ with respect to the basis $\{v_1,\ldots,v_n\}$.  We consider
the associated map $g:=\bigwedge^d(f) \in \End_k(\bigwedge^d(V))$. Given $(K',K) \in \cS(d)^2$, we let $A(f)_{(K',K)}$ be the $(K',K)$-minor of $A$. Then we have
\[ g(v_K) = \sum_{K'\in\cS(d)}\det(A(f)_{(K',K)})v_{K'} \ \ \ \ \ \ \ \ \forall \ K \in \cS(d).\]
Thus, the coordinates of $\overline{\ev}([f])$ for $[f] \in \cU_J$ are $[\det(A(f)_{(I,J)})]_{I\in\cS(d)}$. As a result, the map 
\[\overline{\ev} : \PP(\End_k(V))_d \lra \PP(\bigwedge^d(V)) \ \ ; \ \ [f] \mapsto \bigwedge^d(f(V))\]
is locally given by the homogeneous polynomials $(\det((X_{ij})_{(I,J)}))_{I\in \cS(d)}$, where $\det((X_{ij})_{(I,J)}) \in k[X_{ij} \ ; \ 1 \le i,j \le n]_d$. 

\bigskip
 
\begin{Example} {\it Let $V$ be a vector space of dimension $n\ge 2$. Then the morphism 
\[ \overline{\ev} : \PP(\End_k(V))_1 \lra \PP(V)\]
is not homogeneous.} 

\medskip
\noindent
Choosing a basis of $V$, we consider $(n\!\times\!n)$-matrices and observe that 
\[ \overline{\ev}: \PP(\Mat_n(k))_1 \lra \PP^{n-1}\] 
associates to each element $[A] \in  \PP(\Mat_n(k))_1$ its one-dimensional column space. 

For $j \in \{1,\ldots,n\}$, there is a morphism 
\[ \psi_j : \PP^{n-1} \lra \PP(\Mat_n(k))_1\]
sending $[x] \in \PP^{n-1}$ to the class of the matrix, whose $\ell$-th colum is of the form $\delta_{\ell,j}x^{\rm tr}$. We have $\overline{\ev}\circ \psi_j = \id_{\PP^{n-1}}$.

Suppose that there are homogeneous polynomials $f_1,\ldots,f_n \in k[X_{ij} \ ; \ 1\!\le i,j\!\le\!n]$ of degree $d$ that define the morphism $\overline{\ev}$. Then $\overline{\ev}\circ \psi_j $ is also defined
by homogeneous polynomials of degree $d$, and an application of Corollary \ref{MPS4} conjunction with the above identity implies $d= \deg(\overline{\ev}\circ \psi_j)=\deg(\id_{\PP^{n-1}})=1$. We 
thus write
\[ f_\ell = \sum_{i,j=1}^n \alpha_{ij\ell}X_{ij}\]
for $\ell \in \{1,\ldots,n\}$. Since, for every $j \in \{1,\ldots,n\}$, we have 
\[ (x_1\!:\!x_2\!:\cdots :\!x_n) = [\sum_{i=1}^n\alpha_{ij\ell}x_i]_{1\le \ell \le n}\]
for all $(x_1\!:\!x_2\!:\cdots :\!x_n) \in \PP^{n-1}$, it follows from Lemma \ref{MPS2} that there exist $\lambda_j \in k^\times$ with $ \lambda_jX_\ell =\sum_{i=1}^n\alpha_{ij\ell}X_i$. Consequently,
$\alpha_{ij\ell} = \delta_{\ell,i}\lambda_j$, so that
\[ f_\ell = \sum_{j=1}^n \lambda_jX_{\ell j} \ \ \ \ \ \ \ \ \ 1\le \ell \le n.\]
Now let $x:=(x_1,x_2,\ldots,x_n) \in k^n\!\smallsetminus\!\{0\}$ be such that $\sum_{j=1}^n\lambda_j x_j =0$. If $A_x$ denotes the $(n\!\times\!n)$-matrix all whose row vectors are $x$, then 
$[A_x] \in \PP(\Mat_n(k))_1$, while 
\[ (f_1(A_x),\ldots, f_n(A_x)) = 0,\]
a contradiction. \end{Example} 

\bigskip

\begin{Remark} With considerably more effort one can show that the morphism
\[ \overline{\ev} : \PP(\End_k(V))_d \lra \PP(V)\]
is not homogeneous for an arbitrary $d \in \{1,\ldots, n\!-\!1\}$. Since the morphisms of interest will arise via ``composites'' of $\overline{\ev}$ with restrictions of representations $\varrho : A \lra 
\End_k(V)$, this appears to be the reason why our methods work best for those finite group schemes that are analogs of elementary abelian groups. \end{Remark}

\bigskip
\noindent
For future reference, we record a duality between the vector spaces $\bigwedge^d(V)$ and $\bigwedge^d(V^\ast)$. Let $\{\delta_1,\ldots,\delta_n\}$ be the basis of $V^\ast$ that is dual to 
$\{v_1,\ldots,v_n\}$ and denote by $\delta_{I,J}$ the Kronecker symbol of $\cS(d)^2$.

\bigskip

\begin{Lem} \label{Pr3} The following statements hold:
\begin{enumerate}
\item For $I,J \in \cS(d)$, we have
\[ \det((\delta_i(v_j))_{(I,J)}) = \delta_{I,J}. \]
\item The unique bilinear form $(\, ,\,) : \bigwedge^d(V^\ast)\!\times\!\bigwedge^d(V) \lra k$, given by
\[(f_1\wedge f_2\wedge\cdots\wedge f_d,w_1\wedge w_2\wedge\cdots\wedge w_d) \mapsto \det((f_i(w_j)))\]
is non-degenerate. \end{enumerate}\end{Lem}
  
\bigskip

\subsection{The morphisms $\msim^j_M$} 
From now on $k$ is assumed to be an algebraically closed field of characteristic $p>0$. If $\cG$ is a finite group scheme over $k$ with coordinate ring $k[\cG]$, then the dual Hopf algebra $k\cG := k
[\cG]^\ast$ is the {\it algebra of measures on $\cG$}. By general theory, finite-dimensional representations of $\cG$ naturally correspond to finite-dimensional $k\cG$-modules. We shall henceforth 
write $\modd \cG$ for the category of finite-dimensional $\cG$-modules and use both interpretations interchangeably. We refer to \cite{Wa} for these matters.

If $M$ is a finite-dimensional $\cG$-module and $x \in k\cG$, we let
\[ x_M : M \lra M \ \ ; \ \ v \mapsto x\dact v\]
be the linear transformation effected by $x$.

Given $r \in \NN$, we denote by $\GG_{a(r)}$ the $r$-th Frobenius kernel of the additive group $\GG_a = \Spec_k(k[T])$. We consider an infinitesimal group scheme $\cG$ of height $r$ along with the 
set $V(\cG)$ of its one-parameter subgroups. By definition, the elements of $V(\cG)$ are homomorphisms $\varphi : \GG_{a(r)} \lra \cG$ of group schemes. General theory then shows that
\[ V(\cG)=\Hom_{\rm Hopf}(k\GG_{a(r)},k\cG) \subseteq \Hom_k(k\GG_{a(r)},k\cG):=\fM\]
is an affine variety. In fact, $V(\cG)$ is the variety of $k$-rational points of the scheme of infinitesimal one-parameter subgroups introduced in \cite{SFB1}.

Recall that the diagonalizable group $k^\times$ acts on $\GG_{a(r)}$ via automorphisms. The action on $\GG_{a(r)}$ corresponds to the operation of $k^\times$ on the coordinate ring $k[\GG_{a(r)}]=k
[T]/(T^{p^r})$ that is given by $\alpha\dact t^i := \alpha^it^i$, where $t:= T\!+\!(T^{p^r})$. Consequently, $k^\times$ operates on $k\GG_{a(r)}=k[\GG_{a(r)}]^\ast$ via automorphisms of Hopf algebras and 
with set of weights $\{-(p^r\!-\!1),\ldots,0\}$. We endow $\fM$ with the structure of a $k^\times$-module via
\[ (\alpha\dact f)(u) := f(\alpha^{-1}.u) \ \ \ \ \ \ \ \ \forall \ \alpha \in k^\times, \, f \in \fM,\, u \in k\GG_{a(r)},\]
so that $\fM = \bigoplus_{i=0}^{p^r-1}\fM_i$, with weight spaces
\[\fM_i = \{f \in \fM \ ; \ f((k\GG_{a(r)})_j)=(0) \ \text{for} \ j\ne -i\} \cong \Hom_k((k\GG_{a(r)})_{-i},k\cG).\]
Note that $V(\cG)$ is $k^\times$-stable and that $V(\cG)_0 = \{\varepsilon\}$, where $\varepsilon : k\GG_{a(r)} \lra k \subseteq k\cG$ is the co-unit. Let $u_{r-1} \in k\GG_{a(r)}$ be the linear
map given by $u_{r-1}(t^j)=\delta_{j,p^{r-1}}$ for $0\le j \le p^r\!-\!1$, so that $u_{r-1} \in k\GG_{a(r)}$ is homogeneous of degree $-p^{r-1}$.

Let $j \in \{1,\ldots,p\!-\!1\}$. Given a $\cG$-module $M$, we let
\[ \rk^j(M) := \max\{ \rk(\varphi(u_{r-1})^j_M) \ ; \ \varphi \in V(\cG)\}\]
be the {\it generic $j$-rank} of $M$. The number $\rk(M):=\rk^1(M)$ is called the {\it generic rank} of $M$. 

For a vector space $V$, we let $\Gr_0(V)$ be the variety consisting of one point. 

\bigskip

\begin{Thm} \label{Mo1} Let $M$ be a $\cG$-module such that $\rk^j(M)=d_j$. Then the following statements hold:
\begin{enumerate}
\item $U_{M,j} := \{ [\varphi] \in \Proj(V(\cG)) \ ; \ \rk(\varphi(u_{r-1})_M^j) = d_j\}$ is an open subset of $\Proj(V(\cG))$.
\item The map
\[ \msim^j_M : U_{M,j} \lra \Gr_{d_j}(M) \ \ ; \ \  [\varphi] \mapsto \im \varphi(u_{r-1})^j_M\] 
is a morphism of varieties. \end{enumerate} \end{Thm}

\begin{proof} (1) Recall that $k\cG$ has the structure of a conical variety via the restriction of the scalar multiplication. We consider the evaluation map
\[ u_{r-1}^\ast : V(\cG) \lra k\cG \ \ ; \ \ \varphi \mapsto \varphi(u_{r-1}).\]
Observing
\[ u_{r-1}^\ast(\alpha\dact\varphi) = \varphi(\alpha^{-1}.u_{r-1}) = \alpha^{p^{r-1}}\varphi(u_{r-1})=\alpha^{p^{r-1}}u^\ast_{r-1}(\varphi)\]
for all $\alpha \in k^\times$, we conclude that $u_{r-1}^\ast$ is a homogeneous morphism of degree $p^{r-1}$.

Let $\varrho : k\cG \lra \End_k(M)$ be the representation afforded by $M$. By the above, the map
\[ \omega^j : V(\cG) \lra \End_k(M) \ \ ; \ \ \varphi \mapsto \varrho(\varphi(u_{r-1})^j)\]
is a homogeneous morphism of conical affine varieties of degree $jp^{r-1}$. As noted in Section 2.1,
\[ O_{d_j} := \{[f] \in \PP(\End_k(M)) \ ; \ \rk(f) \ge d_j\}\]
is an open subset of $\PP(\End_k(M))$. Since $\omega^j$ is homogeneous, the set
\[ U_{M,j} = \{ [\varphi] \in \Proj(V(\cG)) \ ; \ [\omega^j(\varphi)] \in O_{d_j}\}\]
is open in $\Proj(V(\cG))$. 

(2)  If $d_j=0$, then $\Gr_{d_j}(M)$ is a point. Hence we may assume that $d_j>0$. Lemma \ref{HM1}(2) now implies that
\[ \overline{\omega}^j: U_{M,j} \lra \PP(\End_k(M)) \ \ ; \ \ [\varphi] \mapsto [\varrho(\varphi(u_{r-1})^j)],\]
is a morphism, whose image is contained in $\PP(\End_k(V))_{d_j}$. In view of
\[ \msim^j_M = \overline{\im} \circ \overline{\omega}^j,\]
our assertion follows from Proposition \ref{Pr2}. \end{proof}

\bigskip
\noindent
By combining $\msim^j_M$ with the Pl\"ucker embedding $\mspl_M : \Gr_{d_j}(M) \lra \PP(\bigwedge^{d_j}(M))$, we obtain a morphism
\[ \mspl_M\circ \msim_M^j : U_{M,j} \lra \PP(\bigwedge^{d_j}(M))\]
such that $\mspl_M\circ \msim_M^j = \overline{\rm ev}\circ \omega^j$. The Example of Section 2.1 indicates that this morphism may not be homogeneous.

Given a $\cG$-module $M$ and a one-parameter subgroup $\varphi \in V(\cG)$, we denote by $\varphi^\ast(M)$ the $k\GG_{a(r)}$-module with
underlying $k$-space $M$ and action given by
\[ a\dact m := \varphi(a)m \ \ \ \ \ \ \ \forall \ a \in k\GG_{a(r)}, \, m \in M.\]
In \cite{SFB2} the authors define a scheme, whose variety of $k$-rational points is the {\it rank variety} 
\[ V(\cG)_M := \{ \varphi \in V(\cG) \ ; \ \varphi^\ast(M)|_{k[u_{r-1}]} \ \text{is not projective}\}\]
of $M$. Note that $V(\cG)_M$ is a conical subset of $V(\cG)$. If $V(\cG)_M \subsetneq V(\cG)$, then $\rk^j(M) = (p\!-\!j)\frac{\dim_kM}{p}$ and $U_{M,j} = 
\Proj(V(\cG))\!\smallsetminus\!\Proj(V(\cG)_M)$.  

Following \cite{FPe10}, we say that a $\cG$-module $M$ has {\it constant $j$-rank}, provided $U_{M,j} = \Proj(V(\cG))$. Modules of constant $1$-rank are said to be of {\it constant
rank}. We record the following direct consequence of Theorem \ref{Mo1}.

\bigskip

\begin{Cor} \label{Mo2} Let $M$ be a $\cG$-module of constant $j$-rank $\rk^j(M)=d_j$. Then the map
\[ \msim^j_M : \Proj(V(\cG)) \lra \Gr_{d_j}(M) \ \ ; \ \  [\varphi] \mapsto \im \varphi(u_{r-1})^j_M\]
is a morphism of varieties. \hfill $\square$ \end{Cor}

\bigskip

\begin{Remark} One can equally well consider maps given by cokernels and kernels. \end{Remark}

\bigskip

\section{Morphisms for restricted Lie Algebras} \label{S:MRLA} In this section we consider representations of  restricted Lie algebras. By general theory, these algebras correspond to 
infinitesimal groups of height $1$, a class whose varieties of infinitesimal one-parameter subgroups are particularly tractable.

Throughout this section, $(\fg,[p])$ denotes a finite-dimensional restricted Lie algebra. The finite-dimensional quotient
\[ U_0(\fg) := U(\fg)/(\{x^p\!-\!x^{[p]} \ ; \ x \in \fg\})\]
of the universal enveloping algebra $U(\fg)$ of $\fg$ is referred to as the {\it restricted enveloping algebra} of $\fg$. The representations of the restricted Lie algebra $(\fg,[p])$ are the modules for the 
Hopf algebra $U_0(\fg)$. We refer the reader to \cite{SF} concerning restricted Lie algebras and their representations.

\bigskip

\subsection{Preliminaries}
The conical Zariski closed subset 
\[ V(\fg) := \{ x \in \fg \ ; \ x^{[p]} =0\}\]
is referred to as the {\it nullcone} of $\fg$. We say that $(\fg,[p])$ is {\it $p$-trivial}, provided $V(\fg)=\fg$. By Engel's theorem such a Lie algebra is necessarily nilpotent.

General theory shows that the cone $V(\fg)$ plays the role of $V(\cG)$. Thus, for $j \in \{1,\ldots,p\!-\!1\}$, the {\it generic $j$-rank} of a $U_0(\fg)$-module $M$ is given by
\[ \rk^j(M) := \max\{ \rk(x_M^j) \ ; \ x \in V(\fg)\},\]
and we say that $M$ has {\it constant $j$-rank}, provided $\rk(x^j_M)=\rk(y^j_M)$ for all $x,y \in V(\fg)\!\smallsetminus\!\{0\}$. In this setting, Theorem \ref{Mo1} reads as follows:

\bigskip

\begin{Cor} \label{RLA1} Let $M$ be a $U_0(\fg)$-module of generic $j$-rank $d_j$. Then the following statements hold:
\begin{enumerate}
\item $U_{M,j} := \{ [x] \in \PP(V(\fg)) \ ; \ \rk(x^j_M)=d_j\}$ is an open subset of $\PP(V(\fg))$.
\item The map
\[ \msim^j_M : U_{M,j} \lra \Gr_{d_j}(M) \ \ ; \ \ [x] \mapsto \im x^j_M\]
is a morphism of varieties. \hfill $\square$ \end{enumerate} \end{Cor}

\bigskip
\noindent
Modules of constant $1$-rank are referred to as being of constant rank. Given a $U_0(\fg)$-module $M$, we write $\rk(M):=\rk^1(M)$ as well as $\msim_M:=\msim^1_M$.

\bigskip

\begin{Lem} \label{RLA2} Let $M$ be a $U_0(\fg)$-module of constant rank $d$ such that $\msim_M: \PP(V(\fg)) \lra \Gr_d(M)$ is constant. If $\fg$ contains a non-abelian $p$-trivial subalgebra, then 
$d=0$. \end{Lem}

\begin{proof} Let $\fh \subseteq \fg$ be a non-abelian $p$-trivial subalgebra. Then $U_0(\fh)$ is a local algebra such that $\fh \subseteq \Rad(U_0(\fh))$, and the restriction $N:= M|_{\fh}$ of $M$ is a
 $U_0(\fh)$-module of constant rank $d$. 

We shall show inductively that $\im x_N \subseteq \Rad^n(N)$ for every $n \ge 1$ and $x \in \fh$, the case $n=1$ following from $\fh \subseteq \Rad(U_0(\fh))$. Let $n\ge 2$. By assumption, there exists 
$x_0 \in [\fh,\fh]\!\smallsetminus\!\{0\}$. Writing $x_0 = \sum_{i=1}^\ell [a_i,b_i]$ with $a_i,b_i \in \fh$, the inductive hypothesis yields
\[ \im (x_0)_N \subseteq \sum_{i=1}^\ell (a_i.b_i.N \!+\! b_i.a_i.N) \subseteq \sum_{i=1}^\ell (a_i.\Rad^{n-1}(N)\!+\!b_i.\Rad^{n-1}(N)) \subseteq \Rad^n(N).\]
Since the map $\msim_N$ is constant, the inclusion between the extreme terms holds for every $x \in \fh$.

As there is $n \in \NN$ with $\Rad^n(N)=(0)$, we see that $\im x_N = (0)$ for all $x \in \fh$. As a result, $d = \rk(N)=0$. \end{proof}

\bigskip

\subsection{Lie algebras with smooth nullcones}
Since $V(\fg)$ is a conical variety, it is smooth if and only if it is a subspace of $\fg$. For such Lie algebras, the results of Section \ref{S:MHP} come to bear. 

\bigskip

\begin{Example} If $\fg$ is nilpotent of nilpotency class $\le p$, then, given $x,y \in V(\fg)$, Jacobson's formula \cite[(II.1)]{SF} yields
\[(x\!+\!y)^{[p]}= \sum_{i=1}^{p-1} s_i(x,y) \in \fg^p = (0),\]
so that $V(\fg)$ is a subspace of $\fg$. \end{Example}

\bigskip
\noindent
We begin by recording a direct consequence of a general result on morphisms between projective spaces.

\bigskip

\begin{Lem} \label{LASN1} Suppose that  $V(\fg)$ is a subspace of $\fg$ and let $M$ be a $U_0(\fg)$-module of constant $j$-rank $d_j$. Then the image of the morphism 
\[ \msim^j_M : \PP(V(\fg))\lra \Gr_{d_j}(M)\]
has dimension $\dim \msim^j_M(\PP(V(\fg))) \in \{0,\dim V(\fg)\!-\!1\}$. \end{Lem}

\begin{proof} Recall that the Pl\"ucker embedding $\mspl_M : \Gr_{d_j}(M) \lra \PP(\bigwedge^{d_j}(M))$ is an injective morphism of projective varieties. Thanks to \cite[I.7,Prop.6]{Mu}, the conclusion
of the Lemma holds for the morphism $\mspl_M\circ \msim^j_M$. Hence it is also valid for $\msim^j_M$. \end{proof}

\bigskip
\noindent
We denote by $C(\fg) := \{x \in \fg \ ; \  [x,\fg]=(0)\}$ the {\it center} of $\fg$. We say that a $U_0(\fg)$-module $M$ has the {\it equal images property}, provided 
\[ \im x_M^j = \im y_M^j \ \ \ \ \text{for all} \ x,y \in V(\fg)\!\smallsetminus\!\{0\} \ \text{and} \ j \in \{1,\ldots,p\!-\!1\}.\]
We denote by $\EIP(\fg)$ the full subcategory of $\modd U_0(\fg)$, whose objects are the equal images modules. 

In view of Lemma \ref{LASN1}, the morphism $\msim^j_M$ is constant or its generic fiber has dimension $0$. The following result demonstrates that $\msim_M$ has more restrictive properties. 

\bigskip

\begin{Thm} \label{LASN2} Let $M$ be a $U_0(\fg)$-module of constant rank $d$. Then the following statements hold:
\begin{enumerate}
\item Suppose that  $V(\fg)$ is a subspace. Then the morphism $\msim_M : \PP(V(\fg)) \lra \Gr_d(M)$ is injective or constant.
\item If $V(\fg)\cap C(\fg)\ne (0)$ and $\msim_M$ is constant, then $M \in \EIP(\fg)$. \end{enumerate} \end{Thm}

\begin{proof}  (1) Let $r:= \dim_kV(\fg)$. By assumption, $\PP(V(\fg)) \cong \PP^{r-1}$, so that we have  a morphism
\[ \msim_M : \PP^{r-1} \lra \Gr_d(M).\] 
Suppose that $\msim_M$ is not injective. Then there exist linearly independent elements $x,y \in V(\fg)$ such that $\im x_M = \im y_M$. We consider the subspace $\fu:= kx\oplus ky$ of $V(\fg)$. 
Given $(\alpha,\beta) \in k^2\!\smallsetminus\!\{0\}$, we have
\[ \im (\alpha x\!+\!\beta y)_M \subseteq \im x_M \! +\! \im y_M = \im x_M.\]
Since $M$ has constant rank, we actually have equality. As a result, the morphism $\msim_M$ is constant on $\PP(\fu)=\PP^1 \subseteq \PP^{r-1}$. Owing to Corollary \ref{MPS5}(2), this implies that the 
map $\msim_M$ is constant. 

(2) We show inductively that $\msim^j_M$ is constant for every $j \in \{1,\ldots,p\!-\!1\}$. Let $j>1$ and let $z \in V(\fg)$ be a non-zero central element. Given $x \in V(\fg)\!\smallsetminus\!\{0\}$, the 
inductive hypothesis implies
\[ \im x_M^j = x_M(\im x_M^{j-1}) =   x_M(\im z_M^{j-1}) = z_M^{j-1}(\im x_M) = \im z_M^j.\]
As a result, the map $\msim^j_M$ is constant. 

It follows that $M \in \EIP(\fg)$. \end{proof}

\bigskip

\begin{Remark} Suppose that $V(\fg)$ is a subspace of dimension $\ge2$. If $M$ is a $U_0(\fg)$-module of constant $j$-rank such that $\msim^j_M$ is not constant, then \cite[(4.1.6)]{Sp} shows
that the number of elements of a generic fiber of $\msim^j_M$ is given by the separability degree $[k(V(\fg))\!:\!k(\msim^j_M(V(\fg)))]_s$ of the fields of rational functions. By the above, this field 
extension is purely inseparable if $j=1$. \end{Remark}

\bigskip

\begin{Cor} \label{LASN3} Suppose that $\fg$ is $p$-trivial, and let $M$ be a $U_0(\fg)$-module of constant rank $d$. Then the following statements hold:
\begin{enumerate}
\item The morphism $\msim_M : \PP(\fg) \lra \Gr_d(M)$ is injective or constant.
\item If $\msim_M$ is constant, then $M \in \EIP(\fg)$.
\item If $\msim_M$ is constant and $\fg$ is not abelian, then $M \cong k^{\dim_kM}$. \end{enumerate} \end{Cor}

\begin{proof} (1) This is a direct consequence of Theorem \ref{LASN2}(1). 

(2) Suppose that $\fg \ne (0)$. Since the $p$-trivial Lie algebra $\fg$ is nilpotent, it follows that $C(\fg)\ne (0)$. Theorem \ref{LASN2}(2) now yields the result.

(3) Lemma \ref{RLA2} implies that $d=0$, whence $x_M=0$ for all $x \in \fg$. Consequently, $M$ is a trivial $U_0(\fg)$-module. \end{proof} 

\bigskip
\noindent
Following \cite{CFP}, we refer to an abelian Lie algebra with trivial $p$-map as being {\it elementary}. These algebras appeared in Hochschild's work \cite{Ho} on restricted cohomology, who called them 
strongly abelian. We denote by $\fe_r$ the, up to isomorphism, unique elementary Lie algebra of dimension $r$ and note that the restricted enveloping algebra $U_0(\fe_r)$ is isomorphic (as an 
associative algebra) to the group algebra $kE_r$ of the $p$-elementary abelian group $E_r$ of rank $r$.  As observed in \cite{CFP}, the set $\EE(2,\fg) := \{\fe \subseteq \fg \ ; \ \fe \ \text{is an 
elementary $p$-subalgebra of dimension} \ 2 \}$ is a closed subset of the Grassmannian $\Gr_2(\fg)$.  

\bigskip

\begin{Examples} Theorem \ref{LASN2}(1) may fail for modules of constant $j$-rank. Suppose that $p\ge 3$.
\begin{enumerate}
\item We consider the Lie algebra $\fe_2=kx\!\oplus\!ky$ as well as $M:= U_0(\fe_2)/\Rad^3(U_0(\fe_2))$. Using the standard basis $\{\bar{x}^i\bar{y}^j \ ; \ 0\le i\!+\!j \le 2\}$, we see that $M$ has 
constant $2$-rank $\rk^2(M)=1$. In this case, the morphism 
\[ \msim_M^2 : \PP^1\lra \PP^5 \ \ ; \ \ (\alpha\!:\!\beta) \mapsto (0\!:\!0\!:\!0\!:\!\alpha^2\!:\!2\alpha\beta\!:\!\beta^2)\]
is injective.
\item Let $N:= M/k\bar{x}\bar{y}$. Then $N$ has constant $2$-rank $\rk^2(N)=1$, and the map
\[ \msim_N^2 : \PP^1\lra \PP^5 \ \ ; \ \ (\alpha\!:\!\beta) \mapsto (0\!:\!0\!:\!0\!:\!\alpha^2\!:\!0\!:\!\beta^2)\]
is not injective. \end{enumerate} \end{Examples}

\bigskip
\noindent
The injectivity of the map $\msim^j_M$ in the first example above is a consequence of the following result, which holds for stable modules. Recall that the general linear group $\GL_r(k)$ acts $\fe_r$
via automorphism of restricted Lie algebras such that $\fe_r\!\smallsetminus\!\{0\}$ is an orbit. This implies that $\GL_r(k)$ acts transitively on $\PP(\fe_r)$. Moreover, $\GL_r(k)$ operates on 
$U_0(\fe_r)$ via automorphisms of Hopf algebras. If $M$ is a $U_0(\fe_r)$-module and $g \in \GL_r(\fe_r)$, then $M^{(g)}$ is the $U(\fe_r)$-module with underlying $k$-space and $U_0(\fe_r)$-structure
\[ u\dact m := (g^{-1}.u).m \ \ \ \ \ \forall \ u \in U_0(\fe_r), \ m \in M.\]
We say that $M$ is {\it $\GL_r(k)$-stable} if $M^{(g)}\cong M$ for all $g \in \GL_r(k)$. Since $\GL_r(k)$ acts transitively on $\fe_r\!\smallsetminus\!\{0\}$, every $\GL_r(k)$-stable module has constant
$j$-rank for every $j \in \{1,\ldots,p\!-\!1\}$. 

\bigskip
\noindent

\begin{Prop} \label{LASN4} Let $M$ be a $\GL_r(k)$-stable $U_0(\fe_r)$-module. If $j \in \{1,\ldots,p\!-\!1\}$ is such that $\msim_M^j$ is not constant, then $\msim^j_M$ is injective. \end{Prop}

\begin{proof} By assumption, there exists for every $g \in \GL_r(k)$ an isomorphism $\kappa_g : M^{(g)} \lra M$ of $U_0(\fe_r)$-modules. This implies in particular the identities
\[ (\ast) \ \ \ \ \ \ \ \ \ \   \im (g.x^j)_M = \kappa_g(\im (g.x^j)_{M^{(g)}}) = \kappa_g(\im x^j_M) \ \ \ \ \ \ \forall \ x \in \fe_r, g \in \GL_r(k).\] 
Suppose that $\im y_M^j = \im x^j_M$. Then we have $\im (g.y^j)_M = \kappa_g(\im y^j_M) = \kappa_g(\im x^j_M) = \im (g.x^j)_M$, so that $g.(\msim^j_M)^{-1}(\msim_M^j([x]))
\subseteq (\msim^j_M)^{-1}(\msim_M^j(g.[x]))$ for every $g \in \GL_r(k)$ and $[x]\in \PP(\fe_r)$. Hence $(\msim^j_M)^{-1}(\msim_M^j(g.[x]))$ = $g.(g^{-1}.(\msim^j_M)^{-1}(\msim_M^j(g.[x])))
\subseteq g.(\msim^j_M)^{-1}(\msim_M^j([x]))$, and we have equality. As $\GL_r(k)$ acts transitively on $\PP(\fe_r)$, all fibers of $\msim_M^j$ have the 
same number of elements. Since $\im_M^j$ is not constant, Lemma \ref{LASN1} ensures that all fibers are finite. 

Let $[x]\in \PP(\fe_r)$ and denote by $G_{[x]}$ the stabilizer of $[x]$ in $\GL_r(k)$. Let $[y]$ be another element of the fiber $(\msim_M^j)^{-1}(\msim_M^j([x]))$. For $g \in G_{[x]}$, we have, 
observing $(\ast)$,
\[ \msim_M^j(g.[y]) = \im (g.y^j)_M = \kappa_g(\im y^j_M) = \kappa_g(\im x^j_M) = \im (g.x^j)_M = \im x^j_M = \msim^j_M([x]).\]
Accordingly, the orbit $G_{[x]}.[y]$ is contained in the finite fiber $(\msim_M^j)^{-1}(\msim_M^j([x]))$. 

As $\GL_r(k)$ acts transitively on $\fe_r\!\smallsetminus\!\{0\}$, the stabilizer $G_{[x]}$ is isomorphic to the group of those matrices, whose first column is of the form $\alpha (1,0,\ldots,0)^{\rm tr}$ for some 
$\alpha \in k^\times$. Consequently, $G_{[x]} \cong k^\times \!\times\!(k^{r-1}\!\rtimes\!\GL_{r-1}(k))$ is connected. As $G_{[x]}.[y]$ is finite, we obtain $G_{[x]}.[y] = \{[y]\}$, so that $G_{[x]} \subseteq 
G_{[y]}$. By symmetry, we thus have $G_{[x]}=G_{[y]}$.

Suppose that $[x] \ne [y]$. Then $\{x,y\}$ and $\{x,x\!+\!y\}$ are parts of two bases of $\fe_r$, and we can find $g \in \GL_r(k)=\GL(\fe_r)$ such that $g.x = x$ and $g.y = x\!+\!y$. Since
$[y]\ne [x\!+\!y]$ we have $g \in G_{[x]}\!\smallsetminus\!G_{[y]}$, a contradiction.  As a result, the map $\msim_M^j$ is injective. \end{proof}

\bigskip

\begin{Remark} If $P$ is a projective $U_0(\fe_r)$-module. Then $P$ is $\GL_r(k)$-stable and Proposition \ref{LASN4} implies that $\msim^j_M$ is injective. Since $M_n:=U_0(\fe_r)/\Rad^n(U_0(\fe_r))$
is $\GL_r(k)$-stable, it follows that $\msim^j_{M_n}$ is injective whenever $j<n$. \end{Remark}

\bigskip

\section{The degree of a module}\label{S:DM}
In this section, we introduce new invariants for modules of restricted Lie algebras, called $j$-degrees. We study the behavior of degrees on short exact sequences, and prove a formula
that relates the $j$-degrees of a module and its dual module to its $j$-rank. This yields information concerning Jordan types of self-dual modules along with the computation of the degrees of
a number of modules over elementary Lie algebras.

\bigskip

\subsection{General properties}
Throughout, $(\fg,[p])$ is assumed to be a restricted Lie algebra. In the sequel, we shall be concerned with modules for restricted Lie algebras $\fg$, whose nullcones $V(\fg)$ are subspaces. As noted 
earlier, every $U_0(\fg)$-module $M$ gives rise to morphisms
\[ \mspl_M \circ \msim^j_M : U_{M,j} \lra \PP(\bigwedge^{\rk^j(M)}(M)),\]
where $U_{M,j} \subseteq \PP(V(\fg))\cong \PP^{r-1}$ is a non-empty open subset. The results of Section \ref{S:MHP} motivate the following:

\bigskip

\begin{Definition} Suppose that $V(\fg)$ is a subspace of $\fg$, $j \in \{1,\ldots,p\!-\!1\}$. Let $M$ be a $U_0(\fg)$-module of generic $j$-rank $\rk^j(M)>0$. Then 
\[ \deg^j(M):=\deg(\mspl_M\circ \msim^j_M)\]
is called the {\it $j$-degree} of $M$. We write $\deg(M):= \deg^1(M)$ and refer to $\deg(M)$ as the {\it degree} of $M$. 

If $\rk^j(M)=0$, then we define $\deg^j(M):=0$. \end{Definition}  

\bigskip
\noindent
Recall that a $U_0(\fg)$-module $M$ is said be of {\it constant Jordan type}, provided $M$ has constant $j$-rank for all $j \in \{1,\ldots,p\!-\!1\}$, cf.\ \cite{CFP1}.

\bigskip

\begin{Remarks} (1) \ Directly from the definition we see that a $U_0(\fg)$-module $M$ of constant Jordan type has the equal images property if and only if $\deg^j(M)=0$ for every $j \in \{1,\ldots,p\!-\!1\}
$. If, in addition, $V(\fg)\cap C(\fg)\ne (0)$, then $\deg(M)=0$ already implies $M \in \EIP(\fg)$.  

(2) \ Let $\fe_2 := kx \oplus ky$ be the two-dimensional elementary Lie algebra. Suppose that $M$ is a $U_0(\fe_2)$-module such that $y_M = 0$. Since $\im (\alpha x \!+\! \beta y)^j_M = \im \alpha^j 
x^j_M$ for all $\alpha, \beta \in k$, the maps $\msim^j_M$ are constant, so that $\deg^j(M)=0$ for all $j \in \{1,\ldots,p\!-\!1\}$. \end{Remarks}

\bigskip

\begin{Prop} \label{DGP1} Suppose that $V(\fg) \subseteq \fg$ is a subspace. Let $M$ be a $U_0(\fg)$-module. Then we have 
\[ \deg^j(M) \in \{0,\ldots,j\rk^j(M)\}\] 
for every $j \in \{1,\ldots,p\!-\!1\}$. \end{Prop}

\begin{proof} (1) Let $j \in \{1,\ldots,p\!-\!1\}$ and put $d:= \rk^j(M)$. Without loss of generality, we may assume that $d>0$. We fix a basis $\{v_1,\ldots, v_n\}$ of $M$ and adopt the conventions of Section 2.1. For $J \in \cS(d)$, we consider the open subset
\[ \cU_J := \{ [f] \in \PP(\End_k(M))_d \ ; \ \bigwedge^d(f)(v_J) \ne 0\}\]
of $\PP(\End_k(M))_d$. Thanks to Proposition \ref{Pr2}(1), the map
\[ \overline{\ev} : \PP(\End_k(M))_d \lra \PP(\bigwedge^d(M)) \ \ ; \ \ [f] \mapsto [\bigwedge^d(f(M))]\]
is a morphism such that
\[  \overline{\ev}([f]) = [\bigwedge^d(f)(v_J)] \ \ \ \ \ \ \ \ \text{for all} \ \ [f] \in \cU_J.\]
Let $f \in \End_k(M)_d$ and denote by $A(f)=(a_{ij}) \in {\rm Mat}_n(k)$ the ($n\!\times\!n$)-matrix representing $f$ with respect to the basis $\{v_1,\ldots,v_n\}$.  As noted earlier, we have
\[ \bigwedge^d(f)(v_K) = \sum_{K'\in\cS(d)}\det(A(f)_{(K',K)})v_{K'} \ \ \ \ \ \ \ \ \forall \ K \in \cS(d).\]
Thus, the Pl\"ucker coordinates of $\overline{\ev}([f])$ for $[f] \in \cU_J$ are $[\det(A(f)_{(I,J)})]_{I\in\cS(d)}$. As a result, the map $\overline{\ev}|_{\cU_J}$ is defined by homogeneous polynomials 
of degree $d$.

The canonical map $\varrho_{M,j} : U_{M,j} \lra \PP(\End_k(M))_d \ \ ; \ \ [x] \mapsto [x^j_M]$ is given by homogeneous polynomials of degree $j$. Setting $O_J := \varrho^{-1}_{M,j}(\cU_J)$, we obtain 
an open cover $U_{M,j} = \bigcup_{J \in \cS(d_j)}O_J$ of $U_{M,j}$ such that $\mspl_M\circ\msim^j_M|_{O_J} = \overline{\ev}\circ \varrho_{M,j}|_{O_J}$ is defined by homogeneous polynomials 
$g_{0,J},\ldots, g_{m,J} \in k[X_0,\ldots, X_{r-1}]$ of degree $jd$, whenever $O_J \neq \emptyset$.  

According to Lemma \ref{MPS3}, the morphism $\mspl_M\circ \msim^j_M: U_{M,j} \lra \PP(\bigwedge^d(M))$ is homogeneous, and Lemma \ref{MPS2} yields $\deg^j(M) \le jd$. \end{proof}

\bigskip

\begin{Remark} Let $M$ be a module for an arbitrary restricted Lie algebra $\fg$. The above arguments also show that the morphism
\[ \msim^j_M : U_{M,j} \lra \Gr_{\rk^j(M)}(M)\]
is locally defined by homogeneous polynomials of degree $j\rk^j(M)$. \end{Remark}

\bigskip

\begin{Thm} \label{DGP2} Suppose that $V(\fg) \subseteq \fg$ is a subspace, and let $\fh \subseteq \fg$ be a $p$-subalgebra such that $\dim V(\fh) \ge 2$. If $M$ is a $U_0(\fg)$-module of constant 
$j$-rank, then we have $\deg^j(M)=\deg^j(M|_{\fh})$. \end{Thm}

\begin{proof} The canonical inclusion $\fh \subseteq \fg$ determines a morphism $\iota : \PP(V(\fh)) \hookrightarrow \PP(V(\fg))$ of degree $1$ such that $\mspl_{M|_\fh} \circ \msim^j_{M|_{\fh}} = 
\mspl_M \circ \msim^j_M\circ \iota$. Since $V(\fh) = V(\fg)\cap \fh$ is a subspace of dimension $\ge 2$, Corollary \ref{MPS5} gives rise to
\begin{eqnarray*}
\deg^j(M|_{\fh}) & = & \deg(\mspl_{M|_\fh} \circ \msim^j_{M|_{\fh}}) = \deg(\mspl_M \circ \msim^j_M\circ \iota) = \deg(\mspl_M \circ \msim^j_M)\deg(\iota)\\
& = & \deg(\mspl_M \circ \msim^j_M) = \deg^j(M),\end{eqnarray*}
as desired. \end{proof}

\bigskip

\begin{Cor} \label{DGP3} Suppose that $V(\fg)\subseteq \fg$ is a subspace and that $\EE(2,\fg)\ne \emptyset$. Then a $U_0(\fg)$-module $M$ of constant rank has the equal images property if 
and only if $\deg(M)=0$. \end{Cor}

\begin{proof} Let $\fe \in \EE(2,\fg)$ and suppose that $\deg(M)=0$. Then $\msim_{M|_\fe}$ is constant and Corollary \ref{LASN3} ensures that $M|_\fe$ has the equal images property. As a result,
$\deg^j(M|_\fe)=0$ for every $j \in \{1,\ldots,p\!-\!1\}$. Thanks to Theorem \ref{DGP2}, this also holds for the $j$-degrees of $M$, so that $\msim_M^j$ is constant for $1\!\le\!j\!\le\!p\!-\!1$. Consequently,
$M$ has the equal images property. \end{proof} 

\bigskip

\begin{Remark} In the situation above, the condition $\EE(2,\fg)\ne\emptyset$ does not follows from $\dim V(\fg)\ge 2$. Suppose that $p\ge 3$. Then the Heisenberg Lie algebra $\fg:=kx\!\oplus\!ky\!\oplus\!kz$ with $p$-map $x^{[p]}=0=y^{[p]}$ and $z^{[p]}=z$ has a linear nullcone of dimension $2$, while $\EE(2,\fg)=\emptyset$. \end{Remark}

\bigskip

\begin{Lem} \label{DGP4} Suppose that $(0) \lra N \stackrel{\iota}{\lra} E \stackrel{\pi}{\lra} M \lra (0)$ is a short exact sequence of $U_0(\fg)$-modules, where $(\fg,[p])$ is a restricted Lie algebra such 
that $V(\fg) \subseteq \fg$ is a subspace. For $j \in \{1,\ldots,p\!-\!1\}$, the following statements hold:
\begin{enumerate}
\item If  $\rk^j(E)=\rk^j(M)\!+\!\rk^j(N)$, then $\deg^j(E) \ge \deg^j(M)\!+\!\deg^j(N)$. 
\item If the sequence splits, then $\deg^j(E) = \deg^j(M)\!+\!\deg^j(N)$.
\item If $E$ has constant $j$-rank and $N \subseteq \bigcap_{x \in V(\fg)\!\smallsetminus\!\{0\}}\im x^j_E$, then $M$ has constant $j$-rank satisfying $\rk^j(E)=\rk^j(M)\!+\!\dim_kN$, while $\deg^j(E) = 
\deg^j(M)$.
\item If $E$ has constant $j$-rank and $\sum_{x \in V(\fg)\!\smallsetminus\!\{0\}}\ker x^j_E \subseteq N$, then $N$ has constant $j$-rank satisfying $\rk^j(E)=\rk^j(N)\!+\!\dim_kM$, while 
$\deg^j(E)= \deg^j(N)\!+\!j\dim_kM$.\end{enumerate}\end{Lem}

\begin{proof} Let $\{v_1,\ldots,v_\ell\}$ be a basis of $E$ such that $\{v_1,\ldots, v_n\}$ is a basis of $N$. We shall compute defining sets of polynomials via this basis. Base 
change amounts to composing morphisms by an automorphism of a suitable $\PP^r$, which does not affect the degree.

Let $d_X := \rk^j(X)$ for $X \in \{M,N,E\}$ . As before, we define $\cS(d_E)$ to be the set of $d_E$-element subsets of $[1,\ell]:=\{1,\ldots,\ell\}$ and let $S_N(d_N)$ and $S_M(d_M)$ denote the 
corresponding sets for $[1,n]$ and $[n\!+\!1,\ell]$, respectively. For a subset $J \subseteq [1,\ell]$, we set $J_N := J\cap [1,n]$ as well as $J_M := J\cap [n\!+\!1,\ell]$ and put 
\[ \cT(d_E) := \{ J \in \cS(d_E) \ ; \  J_X \in \cS_X(d_X) \ \text{for} \ X \in \{M,N\}\}.\]
Given elements $v,w \in V$ of a vector space $V$, we shall write $v \approx w$ to indicate that $kv=kw$. 
 
(1) We assume that $d_M,d_N \ne 0$, leaving the requisite modifications of the ensuing arguments for the remaining cases to the reader.

By choice of the basis, the vectors $w_j := \pi(v_j)$ form a basis $\{w_{n+1},\ldots,w_\ell\}$ of $M$. In view of our current assumption, there exists $J \in \cT(d_E)$ such that 
\[O := \{ [x] \in \PP(V(\fg)) \ ; \ \bigwedge^{d_N}(x^j_N)(v_{J_N}) \ne 0 \ \text{and} \ \bigwedge^{d_M}(x^j_M)(w_{J_M}) \ne 0\}\] 
is a non-empty open subset of $U_{N,j}\cap U_{M,j} \subseteq U_{E,j}$.

Let $(f_I)_{I\in \cS_N(d_N)}$, $(g_{I'})_{I' \in \cS_M(d_M)}$ and $(h_{I''})_{I'' \in \cS(d_M)}$ be reduced defining systems for the morphisms $\mspl_N \circ \msim^j_N$, $\mspl_M \circ \msim^j_N$ and 
\[ \zeta_J : O \lra \PP(\bigwedge^{d_M}(E)) \ \ ; \ \ [x] \mapsto \bigwedge^{d_M}(x^j_E)(v_{J_M}),\]
respectively. In view of
\[ \bigwedge^{d_M}(x^j_M)(w_{J_M}) = \bigwedge^{d_M}(\pi)(\zeta_J(v_{J_M})) \approx \sum_{I' \in \cS_M(d_M)}h_{I'}(x)w_{I'} \ \ \ \ \ \ \ \ \ \ \forall \ x \in O\]
there exists $I_0' \in \cS_M(d_M)$ such that $h_{I'_0}\ne 0$. By the same token, Lemma \ref{MPS2}(2) provides a homogeneous polynomial $h$ such that $O \subseteq D(h)$ and 
\[ (\ast) \ \ \ \ \ \ \ \ \ \ h_{I'} = hg_{I'} \ \ \ \ \text{for all} \ \ I' \in \cS_M(d_M).\] 
Let $\cP(d_E) := \{ (I,I'') \in \cS_N(d_N)\!\times\!\cS(d_M) \ ; \ I\cap I'' = \emptyset\}$, so that $\cT(d_E) \subseteq \{I\sqcup I'' \ ; \ (I,I'') \in \cP(d_E)\}$. Given $[x] \in O \subseteq U_{E,j}$, we have
\begin{eqnarray*}
\bigwedge^{d_E}(x^j_E)(v_J) & = & \bigwedge^{d_N}(x^j_N)(v_{J_N})\wedge\bigwedge^{d_M}(x^j_E)(v_{J_M}) \\
& \approx & \sum_{I \in \cS_N(d_N)} f_I(x)v_I\wedge \sum_{I'' \in \cS(d_M)} h_{I''}(x)v_{I''}\\
& = & \sum_{(I,I'') \in \cP(d_E)} \pm f_I(x)h_{I''}(x) v_{I\sqcup I''}.\end{eqnarray*}
Since there exists $I_0 \in \cS_N(d_N)$ such that $f_{I_0}\ne 0$, we conclude that the left-hand side does not vanish on the non-empty open set $O':=O\cap D(f_{I_0})\cap D(h_{I'_0})$.
Consequently, the morphism
\[ \mspl_E \circ \msim^j_E|_{O'} : O'  \lra \PP(\bigwedge^{d_E}(E)) \ \ ; \ \ [x] \mapsto \overline{\ev}([x^j_E])\]
is defined by the polynomials $(\pm f_Ih_{I''})_{(I,I'') \in \cP(d_E)}$ together with zero polynomials. Thus, if $(\gamma_Q)_{Q \in \cS(d_E)}$ is a reduced defining system for $ \mspl_E \circ \msim^j_E$, 
then Lemma \ref{MPS2}(2) provides a homogeneous polynomial $g$ such that
\[ (\ast \ast) \ \ \ \ \ \ \ \ \ g\gamma_{I\sqcup I''} = \pm f_Ih_{I''} \ \ \ \ \ \ \ \ \text{for all} \ (I,I'') \in \cP(d_E).\]
Let $p$ be an irreducible factor of $g$. Then there exist $K_0 \in \cS_N(d_N)$ and $K'_0 \in \cS_M(d_M)$ such that $p\nmid f_{K_0}$ and $p\nmid g_{K'_0}$. Thus, $p\nmid f_{K_0}g_{K'_0}$, while
($\ast$) and ($\ast\ast$) imply $g\gamma_{K_0\sqcup K'_0} = \pm hf_{K_0}g_{K'_0}$.  As a result, $m_p(g) \le m_p(h)$, so that $h = gg'$ for some homogeneous polynomial $g'$. Identities ($\ast$) and ($\ast\ast$) now yield
\[ \gamma_{K_0\sqcup K'_0} = \pm g'f_{K_0}g_{K'_0},\]
whence
\begin{eqnarray*} 
\deg^j(E)& = & \deg(\gamma_{K_0\sqcup K'_0}) = \deg(g')\!+\!\deg(f_{K_0})\!+\!\deg(g_{K'_0}) = \deg(g')\!+\!\deg^j(N)\!+\!\deg^j(M)\\ 
& \ge & \deg^j(N)\!+\!\deg^j(M),
\end{eqnarray*}
as desired.

(2) Since the sequence splits, we have $\rk^j(E)=\rk^j(N)\!+\!\rk^j(M)$. We adopt the notation from (1) and let $(f_I)_{I \in \cS_N(d_N)}$ and $(g_{I'})_{I' \in \cS_M(d_M)}$ be reduced defining systems for 
$\mspl_N\circ \msim^j_N$ and $\mspl_M\circ \msim^j_M$, respectively. Then we have for $x \in O$,
\begin{eqnarray*}
0 \ne \bigwedge^{d_E}(x^j_E)(v_J) & = & \bigwedge^{d_N}(x^j_N)(v_{J_N})\wedge\bigwedge^{d_M}(x^j_M)(v_{J_M}) \\
& \approx & \sum_{I \in \cS_N(d_N)} f_I(x)v_I\wedge \sum_{I' \in \cS_M(d_M)} h_{I'}(x)v_{I'}\\
& = & \sum_{(I,I') \in \cS_N(d_N)\times \cS_M(d_M)} f_I(x)h_{I'}(x) v_{I\sqcup I'}.\end{eqnarray*}
Since the polynomials $(f_Ig_{I'})_{(I,I') \in \cS_N(d_N)\times \cS_M(d_M)}$ have greatest common divisor $1$, our assertion follows.

(3) Let $\sigma : M \lra E$ be a $k$-linear splitting of $\pi$. If $d_E=0$, then $N=(0)$, and there is nothing to be shown. Alternatively, let $x \in V(\fg)\!\smallsetminus\!\{0\}$. Then $\dim_k\im x^j_M = 
\dim_k\im x^j_E\!-\!\dim_k(\im x^j_E\cap N) = \dim_k\im x^j_E\!-\!\dim_kN$, so that $M$ has constant $j$-rank $d_M= d_E\!-\!\dim_kN$. The splitting property implies that
\[ (\dagger) \ \ \ \ \ \ \ \ \ V = \sigma(\pi(V))\oplus N\]
for every subspace $V \subseteq E$ containing $N$. In particular, $E = \im \sigma \oplus N$, and the map $\sigma$ gives rise to an injective $k$-linear map 
\[\omega : \bigwedge^{d_M}(M) \lra \bigwedge^{d_E}(E) \ \ ; \ \ x \mapsto \bigwedge^{d_M}(\sigma)(x)\wedge u,\]
where $u \in \bigwedge^{\dim_kN}(N)\smallsetminus\!\{0\}$. Being linear, $\omega$ induces a morphism $\bar{\omega} : \PP(\bigwedge^{d_M}(M)) \lra \PP(\bigwedge^{d_E}(E))$ of degree $1$.
Let $\lambda : \Gr_{d_M}(M) \lra \Gr_{d_E}(E)$ be defined via $\lambda(W):=\sigma(W)\!\oplus\!N$. In view of ($\dagger$), direct computation reveals that $\bar{\omega}\circ \mspl_M = \mspl_E\circ
\lambda$ as well as $\lambda \circ \msim^j_M=\msim^j_E$, so that 
\[ \bar{\omega} \circ \mspl_M \circ \msim^j_M = \mspl_E\circ\lambda\circ \msim^j_M = \mspl_E\circ \msim^j_E.\]
Owing to Corollary \ref{MPS5}, we thus arrive at
\[ \deg^j(E) = \deg(\mspl_E\circ \msim^j_E) = \deg(\bar{\omega} \circ \mspl_M \circ \msim^j_M) = \deg(\mspl_M \circ \msim^j_M) = \deg^j(M),\]
as asserted.

(4) By assumption, we have $\ker x^j_N = \ker x^j_E$ for all $x \in V(\fg)\!\smallsetminus\!\{0\}$, so that $N$ has constant $j$-rank
\[ \rk^j(N) = \dim_kN\!-\!\dim_k\ker x^j_E = \dim_kE\!-\!\dim_k\ker x^j_E\!-\!\dim_kM = \rk^j(E)\!-\!\dim_kM\]
for all $x \in V(\fg)\!\smallsetminus\!\{0\}$. 

Let $V:=\langle\{v_{n+1},\ldots,v_\ell\}\rangle$, so that $E=N\!\oplus\!V$. If $x \in V(\fg)\!\smallsetminus\!\{0\}$, the condition $\ker x_E^j \subseteq N$ implies 
\[ (\dagger\dagger) \ \ \ \ \ \ \ \ \ \im x_E^j = \im x^j_N\!\oplus\! x^j_E(V)\]
as well as $\dim_kx^j_E(V)=\ell\!-\!n$.

Let $J \in \cS(d_E)$ be such that $O_J:= \{ [x] \in \PP(V(\fg)) \ ; \ \bigwedge^{d_E}(x^j_E)(v_J)\ne 0\} \ne \emptyset$. Given $[x] \in O_J$, we have
\[ 0 \ne \bigwedge^{d_E}(x^j_E)(v_J) = \bigwedge^{|J_N|}(x^j_N)(v_{J_N})\wedge \bigwedge^{|J_M|}(x^j_E)(v_{J_M}).\]
Hence $|J_N|\le d_N$ and ($\dagger\dagger$) implies $|J_M| \le \dim_kM$. Since
\[ |J| = d_E = d_N\!+\!\dim_kM \ \ \text{while} \ \ J = J_N \sqcup J_M,\]
we obtain $|J_N| = d_N$ and $|J_M|=\ell\!-\!n$, so that $J_M = [n\!+\!1,\ell]$. 

Since $x_E^j|_V$ is injective for every $x \in V(\fg)\!\smallsetminus\!\{0\}$, we have $\bigwedge^{\ell-n}(x_E^j)(v_{[n+1,\ell]})\ne 0$ for all $[x] \in \PP(V(\fg))$. There results a morphism
\[ \xi : \PP(V(\fg)) \lra \PP(\bigwedge^{\ell\!-\!n}(E)) \ \ ; \ \ [x] \mapsto \bigwedge^{\ell-n}(x^j_E)(v_{[n+1,\ell]}),\]
whose coordinates are $[\det(A(x_E^j)_{(Q,[n\!+\!1,\ell])})]_{Q \in \cS(n\!-\!\ell)}$. Thus, $\xi$ is defined by homogeneous polynomials of degree $j(\ell\!-\!n)$, so that $\deg(\xi) = j(\ell\!-\!n)$. 

Let $(g_K)_{K\in \cS_N(d_N)}$ and $(h_Q)_{Q \in \cS(\ell\!-\!n)}$ be reduced defining systems for $\mspl_N\circ\msim^j_N$ and $\xi$, respectively.  As before, we put $\cP(d_E)= 
\{(K,Q) \in \cS_N(d_N)\!\times\!\cS(\ell\!-\!n) \ ; \ K\cap Q = \emptyset\}$. Given $[x] \in \PP(V(\fg))$, there exists $J \in \cS(d_E)$ such that $[x] \in O_J$.  The observations above imply
\begin{eqnarray*} 
\mspl_E\circ \im^j_E([x]) & \approx & \bigwedge^{d_E}(x^j_E)(v_J) \approx \sum_{K \in \cS_N(d_N)}\sum_{Q \in \cS(\ell-n)} g_K(x)h_Q(x) v_K\wedge v_Q \\
& = &\sum_{(K,Q) \in \cP(d_E)}\pm g_K(x)h_Q(x) v_{K\sqcup Q}.
\end{eqnarray*}
As a result, the polynomials $(\pm g_Kh_Q)_{(K,Q) \in \cP(d_E)}$ together with zero polynomials constitute a defining system for $\mspl_E\circ \msim^j_E: \PP(V(\fg)) \lra \PP(\bigwedge^{d_E}(E))$. 
Corollary \ref{MPS4} implies that the system is reduced, so that
\[ \deg^j(E) = \deg^j(N)\!+\!\deg(\xi) = \deg^j(N)\!+\!j\dim_kM,\]
as desired. \end{proof}

\bigskip

\begin{Remarks} (1) \ The additivity of the generic $j$-ranks is usually not a consequence of the exactness of the sequence. Consider the exact sequence
\[ (0) \lra k \lra U_0(\fe_r) \lra U_0(\fe_r)/k \lra (0)\]
of modules of constant ranks. Since $\rk(k)=0$ and $\rk(U_0(\fe_r))=p^{r-1}(p\!-\!1)$ while $\rk(U_0(\fe_r)/k)=p^{r-1}(p\!-\!1)\!-\!1$, we have $\rk(U_0(\fe_r))> \rk(U_0(\fe_r)/k)\!+\!\rk(k)$.

(2) \ The additivity of generic $j$-ranks holds whenever the sequence is {\it locally split}, i.e.\ when the restricted sequence
\[ (0) \lra N|_{k[x]} \lra E|_{k[x]} \lra M|_{k[x]} \lra (0)\]
is split exact for every $x \in V(\fg)$. Here $k[x] = U_0(kx)$ is the subalgebra of $U_0(\fg)$ generated by $x$.

(3) \ Lemma \ref{DGP4}(2) does not hold for arbitrary exact sequences of modules, whose generic $j$-ranks are additive. We consider the Lie algebra $\fe_2 := kx\oplus ky$ as well as the 
$U_0(\fe_2)$-module $E:= U_0(\fe_2)/\Rad^2(U_0(\fe_2))$ with its canonical basis $\{\bar{1},\bar{x},\bar{y}\}$. Let $N:= k\bar{y}$. Then we have $\rk(E)=1=\rk(E/N)$, while $\rk(N)=0$. The module $E$ 
has degree $\deg(E) = 1$, while $y_{E/N}=0$ and $y_N=0$ imply $\deg(N)=0=\deg(E/N)$. 

(4) \ If $M$ is a $U_0(\fe_r)$-module of constant $j$-rank, then $\tilde{\fK}^j(M):=\sum_{x \in \fe_r\!\smallsetminus\!\{0\}}\ker x_M^j$ is a submodule of $M$ and Lemma \ref{DGP4}(4) shows
that $\tilde{\fK}^j(M)$ has constant $j$-rank while $\deg^j(M) = \deg^j(\tilde{\fK}^j(M))\!+\!j\dim_kM/\tilde{\fK}^j(M)$. In particular, we have $M=\tilde{\fK}^j(M)$, whenever $\deg^j(M)<j$.\end{Remarks}

\bigskip

\begin{Remark} Let $(\fg,[p])$ be a restricted Lie algebra such that $\zeta : \PP^1 \lra \PP(\fg)$ is a non-constant morphism that factors through $\PP(V(\fg))$. If $M$ is a $U_0(\fg)$-module of
constant $j$-rank, one may consider $\deg^j_\zeta(M) := \deg(\mspl_M\circ \msim_M^j\circ \zeta)$. For $\fg = \fsl(2)$, the ``Veronese embedding'' 
\[ \zeta : \PP^1 \lra \PP(\fsl(2)) \ \ ; \ \ (x\!:\!y) \mapsto [\left(\begin{smallmatrix} xy & x^2 \\ -y^2 & -xy\end{smallmatrix}\right)]\] 
is a morphism of degree $\deg(\zeta)=2$ such that $\im \zeta = \PP(V(\fsl(2)))$ and $\PP^1\lra \PP(V(\fsl(2)))$ is an isomorphism. \end{Remark}

\bigskip

\subsection{The rank-degree formula}\label{S:RDF}
Let $(\fg,[p])$ be a restricted Lie algebra, $M$ be a $U_0(\fg)$-module of dimension $n$. For $d\le n$, the assignment
\[ (\, , \,) : \bigwedge^d(M^\ast)\!\times\!\bigwedge^d(M) \lra k \ \ ; \ \ (f_1\wedge f_2\wedge \cdots\wedge f_d, m_1\wedge m_2\wedge \cdots\wedge m_d) \mapsto \det((f_i(m_j)))\]
defines a non-degenerate bilinear form. The following subsidiary result shows that this form is compatible with the actions of $U_0(\fg)$ on $M$ and $M^\ast$. 

\bigskip

\begin{Lem} \label{RDF1} Let $M$ be an $n$-dimensional $U_0(\fg)$-module, $d \le n$. For $j \in \{1,\ldots,p\!-\!1\}$. We have 
\[(\bigwedge^d(x^j_{M^\ast})(a),m) = (-1)^{jd}(a,\bigwedge^d(x^j_M)(m))\] 
for all $x \in \fg$, $a \in \bigwedge^d(M^\ast), m \in \bigwedge^d(M)$.\end{Lem}

\begin{proof} It suffices to verify the assertion for $a=f_1\wedge f_2 \wedge \cdots \wedge f_d$, and $m=m_1\wedge m_2 \wedge \cdots \wedge m_d$. For $x \in \fg$ and $\ell \in \{1,\ldots,p\!-\!1\}$
we have
\begin{eqnarray*}
(\bigwedge^d(x^\ell_{M^\ast})(f_1\wedge \cdots \wedge f_d), m_1\wedge \cdots \wedge m_d) & = & (x^\ell\dact f_1\wedge \cdots \wedge x^\ell \dact f_d, m_1\wedge \cdots \wedge m_d) \\ 
&= &\det(((x^\ell\dact f_i)(m_j)))\\ 
& = &\det(((-1)^\ell f_i(x^\ell\dact m_j))) = (-1)^{\ell d} \det((f_i(x^j\dact m_j))) \\
&= & (-1)^{\ell d} (f_1\wedge \cdots \wedge f_d, \bigwedge^d(x^\ell_M)(m_1\wedge \cdots \wedge m_d)),
\end{eqnarray*} 
as desired. \end{proof}

\bigskip

\begin{Thm} \label{RDF2} Suppose that $V(\fg)\subseteq \fg$ is a subspace. Let $M$ be a $U_0(\fg)$-module of constant $j$-rank. Then we have
\[ \deg^j(M)\!+\! \deg^j(M^\ast) = j\rk^j(M).\]
\end{Thm}

\begin{proof} We put $d:=\rk^j(M)$. If $d=0$, then $\deg^j(M)=0=\deg^j(M^\ast)$. We therefore assume $d>0$. Adopting our previous notation, we let $\{v_1,\ldots,v_n\}$ be a basis of $M$, 
with dual basis $\{\delta_1,\ldots, \delta_n\} \subseteq M^\ast$. Let $A(x_M^j) \in \Mat_n(k)$ be the $(n\!\times\!n)$-matrix representing $x_M^j$ relative to $\{v_1,\ldots,v_n\}$. Given $I,J \in \cS(d)$, 
Lemma \ref{RDF1} implies
\[ (\ast) \ \ \ \ \ \ \ \ ( \bigwedge^d(x^j_{M^\ast})(\delta_I),v_J) = (-1)^{jd} (\delta_I,\bigwedge^d(x^j_M)(v_J)) = (-1)^{jd}\det(A(x_M^j)_{(I,J)})\]
for every $x \in \fg$. 

Now let $(f_I)_{I\in \cS(d)}$ and $(g_I)_{I\in \cS(d)}$ be reduced defining systems for the morphisms $\mspl_{M^\ast}\circ \msim^j_{M^\ast}$ and $\mspl_M\circ \msim^j_M$, respectively. 
Given $I \in \cS(d)$, we consider the open set $O_I := \{[x] \in \PP(V(\fg)) \ ; \ \bigwedge^d(x^j_{M^\ast})(\delta_I)\ne 0\}$. Then we have
\[ \bigwedge^d(x^\ell_{M^\ast})(\delta_I) \approx \sum_{K\in \cS(d)} f_K(x)\delta_K \ \ \ \ \ \ \ \ \ \ \forall \ [x] \in O_I.\]
Consequently, Lemma \ref{Pr3} in conjunction with ($\ast$) implies
\[ (\ast \ast) \ \ \ \ \ \ \ \ [f_J(x)]_{J \in \cS(d)} = [\det(A(x^j_M)_{(I,J)})]_{J \in \cS(d)} \ \ \ \ \ \ \ \ \ \ \forall \ [x] \in O_I,\]
where the elements belong to $\PP^{\binom{n}{d}-1}$. Thus, we have two defining systems for the morphism 
\[ \mspl_{M^\ast}\circ \msim^j_{M^\ast}|_{O_I}: O_I \lra \PP^{\binom{n}{d}-1}.\]
Note that the morphism $V(\fg) \lra k \ ; \ x \mapsto \det(A(x_M^j)_{(I,J)})$ is given by a homogeneous polynomial $\gamma_{(I,J)}$ of degree $jd$. Thus, if $O_I \neq \emptyset$, then 
Lemma \ref{MPS2} provides a homogeneous polynomial $h_I$ with $O_I \subseteq D(h_I)$ such that
\[ (\ast \ast\ast) \ \ \ \ \ \ \ \ \gamma_{(I,J)} = h_I f_J \ \ \ \ \ \ \ \ \ \forall \ J \in \cS(d).\] 
For $O_I=\emptyset$, identity ($\ast$) implies $\gamma_{(I,J)}=0$, so that the choice $h_I=0$ gives the same identity.

For $J \in \cS(d)$, we consider $U_J := \{ [x] \in \PP(V(\fg)) \ ; \ \bigwedge^d(x^j_M)(v_J)\ne 0\}$. Let $J \in \cS(d)$ be such that $U_J\ne \emptyset$. Given $[x] \in U_J$, we have
\[ (\gamma_{(I,J)}(x))_{I\in \cS(d)} \ne 0,\]
whence $f_J(x)\ne 0$. Moreover,
\[ [\gamma_{(I,J)}(x)]_{I \in \cS(d)} = [h_I(x)f_J(x)]_{I\in \cS(d)} = [h_I(x)]_{I\in \cS(d)}  \ \ \ \ \ \ \ \ \forall \ [x] \in U_J,\]
implying that $(h_I)_{I\in \cS(d)}$ and $(g_I)_{I\in \cS(d)}$ are defining systems for the morphism $\mspl_M\circ\msim^j_M|_{U_J}$, with the latter being reduced. Consequently, Lemma \ref{MPS2} furnishes a homogeneous polynomial $\tilde{h}_J$ such that $U_J \subseteq D(\tilde{h}_J)$ and
\[ h_I = \tilde{h}_Jg_I \ \ \ \ \ \ \ \ \forall \ I \in \cS(d).\] 
Thus, if $U_J,U_{J'} \neq \emptyset$, then $\tilde{h}:= \tilde{h}_J=\tilde{h}_{J'}$, so that $U_{J'} \subseteq D(\tilde{h})$. Since $\PP(V(\fg))=\bigcup_{J\in \cS(d)}U_J$, we conclude that $\PP(V(\fg)) \subseteq D(\tilde{h})$, implying that $\tilde{h}$ is constant.

\medskip
Since the module $M$ has constant $j$-rank $d>0$, we obtain a morphism 
\[ \Phi : \PP(V(\fg)) \lra \PP^{\binom{n}{d}^2-1} \ \ ; \ \ [x] \mapsto [\det(A(x^j_M)_{(I,J)})]_{(I,J)\in \cS(d)^2}\]
of degree $jd$. It now follows from $(\ast\ast\ast)$ that $(g_If_J)_{(I,J) \in \cS(d)^2}$ is a reduced defining system for $\Phi$, so that
\[ j\rk^j(M)= \deg(\Phi) = \deg(g_I)\!+\!\deg(f_J) = \deg^j(M)\!+\!\deg^j(M^\ast)\] 
for every pair $(I,J) \in \cS(d)^2$ such that $g_If_J\ne 0$. \end{proof} 

\bigskip
\noindent
Let $M$ be a $U_0(\fg)$-module. Then $M$ has the {\it equal $j$-images property} if $\msim_M^j$ is constant. We say that $M$ has the {\it equal $j$-kernels property}, provided there exists a subspace 
$W_j \subseteq M$ such that $\ker x^j_M = W_j$ for all $x \in V(\fg)\!\smallsetminus\!\{0\}$. If $M$ has the equal $j$-kernels property for all $j \in \{1,\ldots,p\!-\!1\}$, then $M$ has the {\it equal kernels 
property}. The full subcategory of equal kernels modules will be denoted $\EKP(\fg)$. 

\bigskip

\begin{Cor} \label{RDF3} Suppose that $V(\fg) \subseteq \fg$ is a subspace, and let $M$ be a $U_0(\fg)$-module of constant $j$-rank. Then $M$ has the equal $j$-kernels property if and only if 
$\deg^j(M)=j\rk^j(M)$. \end{Cor}

\begin{proof} A module $M$ has the equal $j$-kernels property if and only if its dual module $M^\ast$ has the equal $j$-images property. Since the latter property is equivalent to $\deg^j(M^\ast)=0$, our
 assertion is a direct consequence of Theorem \ref{RDF2}. \end{proof}

\bigskip

\begin{Example} We consider the elementary Lie algebra $\fe_r$ and choose a basis $\{x_1,\ldots,x_r\} \subseteq \fe_r$. Using multi-index notation, we put $\tau:=(p\!-\!1,\ldots,p\!-\!1) \in \NN^r_0$ 
as well as $\epsilon_i := (\delta_{ij})_{1\le j \le r}$ for $i \in \{1,\ldots, r\}$. Suppose that $p\ge 3$. Setting $v_r := \sum_{i=1}^r x^{\tau-2\epsilon_i} \in U_0(\fe_r)$, we consider the submodule 
$M_{r+2}:=kv_r\oplus\Soc_2(U_0(\fe_r))$ of $U_0(\fe_r)$. This module has Loewy length $\ell\ell(M_{r+2})=3$. 

We write $M_{r+2} = kv_r \oplus \bigoplus_{i=1}^r kx_i.v_r \oplus kx^\tau$ and note that $x_ix_jv_r = \delta_{ij}x^\tau$. For $\lambda = (\lambda_1,\ldots,\lambda_r) \in k^r\!\smallsetminus\!\{0\}$, we 
put $x(\lambda) := \sum_{i=1}^r\lambda_ix_i$. Then we have 
\begin{itemize}
\item $\im x(\lambda)_{M_{r+2}} = k(\sum_{i=1}^r\lambda_ix_i.v_r)\oplus kx^\tau$, as well as
\item $\im x(\lambda)_{M_{r+2}}^2 = k(\sum_{i=1}^r\lambda_i^2)x^\tau$. \end{itemize}
Hence, for $r\ge 2$, the module $M_{r+2}$ has constant rank $\rk(M_{r+2})= 2$, but not constant $2$-rank. As a result, $M_{r+2}$ is neither an equal $1$-images module nor an equal $1$-kernels module, so that $\deg(M_{r+2})=1$. \end{Example}

\bigskip

\subsection{Applications: Self-dual modules and exact sequences}
Given $x \in V(\fg)\!\smallsetminus\!\{0\}$, the subalgebra $k[x] \subseteq U_0(\fg)$ is isomorphic to the truncated polynomial ring $k[T]/(T^p)$. Setting $[i]:= k[x]/(x^i)$ for $1\le i \le p$, we obtain
a full set $\{[1],\ldots,[p]\}$ of representatives for the isomorphism classes of the indecomposable $k[x]$-modules. If $M$ is a $U_0(\fg)$-module, then 
\[ M|_{k[x]} \cong \bigoplus_{i=1}^pa_i(x)[i]\]
for some $a_i(x) \in \NN_0$. This isomorphism class is the {\it Jordan type $\Jt(M,x)$ of $M$ relative to $x$}. We denote by
\[ \Jt(M)=\{\Jt(M,x) \ ; \ x \in V(\fg)\!\smallsetminus\!\{0\}\}\]
the finite set of Jordan types of $M$. Thus, the $U_0(\fg)$-module $M$ has constant Jordan type if and only if $|\Jt(M)|=1$. In that case, we write
\[ \Jt(M)= \bigoplus_{i=1}^pa_i[i].\]

\bigskip

\begin{Remark} Let $M$ be an equal images module of constant rank $\rk(M)\ne 0$. Then $M^\ast$ is an equal kernels module of constant rank $\rk(M^\ast) = \rk(M)$. By the above, we have $\deg(M)=0\ne \deg(M^\ast)$, while $\Jt(M)=\Jt(M^\ast)$. Thus, $\deg(M)$ discerns properties of modules that cannot be detected by means of their Jordan types. \end{Remark} 

\bigskip
\noindent
In the sequel, we shall be concerned with self-dual modules, that is, $U_0(\fg)$-modules $M$ satisfying $M\cong M^\ast$. Such a module $M$ can be characterized via the existence of a non-
degenerate invariant bilinear form $(\, , \,) : M\!\times\!M \lra k$.

Let $\eta$ denote the {\it antipode} of the Hopf algebra $U_0(\fg)$, that is, the unique anti-automorphism of $U_0(\fg)$ such that $\eta(x)=-x$ for all $x \in \fg$. A bilinear form $(\, , \,) : M\!\times\!M \lra k$
is referred to as {\it invariant}, if
\[ (a.m,m') = (m,\eta(a).m') \ \ \ \ \ \ \ \ \ \forall \ a \in U_0(\fg), \, m,m' \in M.\]
We denote by ${}^\perp V$ and $V^\perp$ the left and right perpendicular spaces of a subspace $V\subseteq M$. If the form $(\, , \,)$ is non-degenerate, then $\dim_kM=\dim_kV\!+\!\dim_kV^\perp =
\dim_kV\!+\!\dim_k{}^\perp V$, so that ${}^\perp(V^\perp)= V = ({}^\perp V)^\perp$.

\bigskip
\begin{Example} By general theory, the Hopf algebra $U_0(\fg)$ is a Frobenius algebra. Accordingly, $U_0(\fg)$ possesses a non-degenerate bilinear form $\langle\, , \, \rangle : 
U_0(\fg)\!\times\!U_0(\fg)\lra k$ such that
\[ \langle ab, c\rangle = \langle a,bc\rangle \ \ \ \ \ \ \ \ \ \ \forall \ a,b,c \in U_0(\fg).\]
Hence the form $(\, , \,) : U_0(\fg)\!\times\!U_0(\fg) \lra k$, given by
\[ (a,b):= \langle \eta(a),b\rangle\]
for all $a,b \in U_0(\fg)$, is a non-degenerate invariant form of the regular module $U_0(\fg)$. \end{Example}
   
\bigskip

\begin{Thm} \label{SD1} Suppose that $\EE(2,\fg)\ne \emptyset$. If $M$ is a self-dual $U_0(\fg)$-module, then the following statements hold:
\begin{enumerate}
\item If $M$ has constant $j$-rank, then $\rk^j(M)\equiv 0 \modd(2)$, whenever $j \equiv 1\modd(2)$.
\item If $M$ has constant Jordan type $\Jt(M)=\bigoplus_{i=1}^pa_i[i]$, then $a_i\equiv 0 \ \modd(2)$ whenever $i\equiv 0 \ \modd(2)$. \end{enumerate}\end{Thm}

\begin{proof} Let $\fe \in \EE(2,\fg)$. We consider the self-dual $U_0(\fe)$-module $N:=M|_\fe$.

(1) Since $N\cong N^\ast$, Theorem \ref{RDF2} implies $2\deg^j(N)=j\rk^j(N)=j\rk^j(M)$.

(2) Note that $N$ has constant Jordan type $\Jt(N)=\Jt(M)$. Setting $\rk^0(N)=\dim_kN$ and $\rk^p(N)=0=\rk^{p+1}(N)$, we have $\rk^j(N)= \sum_{i=j+1}^p a_i(i\!-\!j)$, so that
\[ \rk^{j-1}(N)\!-\!\rk^j(N) = \sum_{i\ge j}a_i.\]
As a result,
\[ a_j = \rk^{j-1}(N)\!-\!2\rk^j(N)\!+\!\rk^{j+1}(N).\]
Let $j \in \{1,\ldots,p\}$ be even. Then $j\!-\!1$ and $j\!+\!1$ are odd, and (1) shows that $\rk^{j-1}(N)$ and $\rk^{j+1}(N)$ are even. Consequently, $a_j$ is even.\end{proof} 

\bigskip

\begin{Remarks} (1) \ Let $i \in \{0,\ldots,p\!-\!1\}$. Since the simple $U_0(\fsl(2))$-module $L(i)$ is self-dual and of constant Jordan type $\Jt(L(i))=[i\!+\!1]$, Theorem \ref{SD1} may fail
if $\EE(2,\fg)=\emptyset$, even though $\PP(V(\fg))$ is isomorphic to a projective space. 

(2) \ The proof of Theorem \ref{SD1} shows that the conclusions are valid for those $U_0(\fg)$-modules whose restriction $M|_\fe$ is self-dual for some $\fe \in \EE(2,\fg)$. 

(3) \ Suppose that $p\ge 3$. The $U_0(\fe_2)$-module $H(\fe_2):= \Rad(U_0(\fe_2))/\Soc(U_0(\fe_2))$ is indecomposable, self-dual, and of constant Jordan type $\Jt(H(\fe_2)) = 2[p\!-\!1]
\!\oplus\!(p\!-\!2) \![p]$. Its second Heller shift $M:= \Omega^2_{U_0(\fe_2)}(H(\fe_2))$ has constant Jordan type $\Jt(M)= 2[p\!-\!1]\!\oplus\!n[p]$ for some $n \in \NN_0$ and thus also fulfills the 
conclusion of Theorem \ref{SD1}. The assumption $M \cong M^\ast$ implies, $\Omega^2_{U_0(\fe_2)}(H(\fe_2)) \cong \Omega^2_{U_0(\fe_2)}(H(\fe_2))^\ast \cong 
\Omega^{-2}_{U_0(\fe_2)}(H(\fe_2)^\ast) \cong \Omega^{-2}_{U_0(\fe_2)}(H(\fe_2))$, so that $\Omega^4_{U_0(\fe_2)}(H(\fe_2))\cong H(\fe_2)$. This shows that $H(\fe_2)$ is periodic, which 
contradicts $H(\fe_2)$ being a module of constant Jordan type. 

(4) \ Suppose that $V(\fg)$ is a subspace of $\fg$. If $M$ is self-dual and of constant $j$-rank, then Theorem \ref{RDF2} also implies $\deg^j(M)\equiv 0 \ \modd(j)$, whenever $j$ is odd. \end{Remarks}  

\bigskip

\begin{Examples} Suppose that $\fg$ is a Lie algebra of dimension $r$ such that $V(\fg)$ is a subspace of dimension $\ge 2$. We fix $j \in \{1,\ldots,n\}$.

(1) \ Since $U_0(\fg)$ is self-dual and of constant $j$-rank $p^{r-1}(p\!-\!j)$, Theorem \ref{RDF2} gives
\[ \deg^j(U_0(\fg))= \frac{j(p\!-\!j)}{2}p^{r-1}.\]

(2) \ {\it We have 
\[ \deg^j(M) = \deg^j(U_0(\fg))\!-\!j \ \ \text{for} \ \ M \in \{\Rad(U_0(\fg)), \Rad(U_0(\fg))/\Soc(U_0(\fg))\}.\] }
In view of $\Rad(U_0(\fg))\cong \Rad(U_0(\fg)^\ast) \cong (U_0(\fg)/\Soc(U_0(\fg)))^\ast$, Theorem \ref{RDF2} and Lemma \ref{DGP4}(3) yield
\begin{eqnarray*} 
\deg^j(\Rad(U_0(\fg))) &= & j\rk^j(U_0(\fg)/\Soc(U_0(\fg)))\!-\!\deg^j(U_0(\fg)/\Soc(U_0(\fg)))\\ 
& = & j\rk^j(U_0(\fg))\!-\!j-\!\deg^j(U_0(\fg))= \deg^j(U_0(\fg))\!-\!j.
\end{eqnarray*}
By the same token, we have $\deg^j(\Rad(U_0(\fg))/\Soc(U_0(\fg)))=\deg^j(\Rad(U_0(\fg)))$.

(3) \ {\it Suppose that $\EE(2,\fg)\ne \emptyset$. Then we have
\[ j\rk(P) = 2\deg^j(P)\]
for every projective $U_0(\fg)$-module $P$}.

Let $\fe \in \EE(2,\fg)$. For a projective $U_0(\fg)$-module $P$, there exists $\ell \in \NN_0$ such that $P|_{\fe} \cong \ell U_0(\fe)$. Since $P$ has constant $j$-rank, Lemma \ref{DGP4}(2), 
Theorem \ref{RDF2} and Theorem \ref{DGP2} imply
\[ j\rk^j(P) =j \ell \rk^j(U_0(\fe)) = 2\ell \deg^j(U_0(\fe)) = 2\deg^j(P).\]

(4) \ We consider the $U_0(\fe_2)$-modules $M_n := U_0(\fe_2)/\Rad^n(U_0(\fe_2))$ for $1\!\le\! n\! \le\! 2p\!-\!2$. Since each module $M_n$ is also a $\GL_2(k)$-module, and $\GL_2(k)$ acts on 
$\PP(\fe)$ with one orbit, it has constant $j$-rank.

If $n\! \le\! p\!-\!1$, then $M_n$ is an equal kernels module, so that Corollary \ref{RDF3} yields 
\[ \deg^j(M_n)=j\rk^j(M_n)= \left\{ \begin{array}{cl} 0 & n\le j \\ j\frac{(n-j)(n-j+1)}{2} & n \ge j\!+\!1.\end{array}\right.\] 

Let $n\!\ge\! p$. Since $\Rad^{p-1}(U_0(\fe_2))$ has the equal images property, it follows that $\Rad^{p-1+j}(U_0(\fe_2))= \im x^j_{\Rad^{p-1}(U_0(\fe_2))}$ for all $x \in \fe_2\! \smallsetminus\!\{0\}$, so 
that $\Rad^n(U_0(\fe_2)) \subseteq \bigcap_{x \in \fe_2\smallsetminus\{0\}}\im x_{U_0(\fe_2)}^j$ for $n\! \ge\! p\!-\!1\!+\!j$. Consequently, Lemma \ref{DGP4}(3) implies 
\[\deg^j(M_n)= \deg^j(U_0(\fe_2))= \frac{pj(p\!-\!j)}{2}\] 
for $n\! \ge\! p\!-\!1\!+\!j$. \end{Examples}

\bigskip

\begin{Cor} \label{SD2} Let $\fg$ be a restricted Lie algebra such that $V(\fg)$ is a subspace of dimension $\ge 2$. If $(0) \lra N \lra E \lra M \lra (0)$ is a short exact sequence of
$U_0(\fg)$-modules of constant $j$-rank with $\rk^j(E)=\rk^j(M)\!+\!\rk^j(N)$, then $\deg^j(E)=\deg^j(M)\!+\!\deg^j(N)$. \end{Cor}

\begin{proof} Dualization provides a short exact sequence $(0) \lra M^\ast \lra E^\ast \lra N^\ast \lra (0)$ of $U_0(\fg)$-modules of constant $j$-rank such that $\rk^j(E^\ast)=\rk^j(M^\ast)\!+\!\rk^j(N^\ast)$.
Lemma \ref{DGP4}(1) yields $\deg^j(E^\ast)\ge \deg^j(M^\ast)\!+\!\deg^j(N^\ast)$, while repeated application of Theorem \ref{RDF2} gives 
\[ \deg^j(E) =  j\rk^j(E)\!-\!\deg^j(E^\ast) \le j\rk^j(M^\ast)\!+\!j\rk^j(N^\ast)\!-\!\deg^j(M^\ast)\!-\!\deg^j(N^\ast) = \deg^j(M)\!+\!\deg^j(N).\]
By applying Lemma \ref{DGP4}(1) to the original sequence, we obtain the reverse inequality. \end{proof}

\bigskip

\begin{Cor} \label{SD3} Let $\fg$ be a restricted Lie algebra such that $V(\fg)$ is a subspace of dimension $\ge 2$. If $(0) \lra N \lra E \lra M \lra (0)$ is a locally split short exact sequence of
$U_0(\fg)$-modules such that $M,N \in \EIP(\fg)$, then $E \in \EIP(\fg)$. \end{Cor}

\begin{proof} Let $j \in \{1,\ldots,p\!-\!1\}$. The equal images modules $M$ and $N$ have constant $j$-rank, and since the sequence is locally split, the $U_0(\fg)$-module $E$ has constant
$j$-rank $\rk^j(E)=\rk^j(M)\!+\!\rk^j(N)$. Corollary \ref{SD2} now implies
\[ \deg^j(E)=\deg^j(M)\!+\!\deg^j(N) =0.\]
As a result, $E$ enjoys the equal images property. \end{proof} 

\bigskip

\section{Low Dimensional Modules of Constant Rank}\label{S:CRLD}
In this section, we will show how Tango's Theorem \ref{Ta1} can be applied in order to obtain information concerning low-dimensional modules of constant rank for finite group schemes.
We begin by considering infinitesimal groups of height $1$. Throughout, $(\fg,[p])$ denotes a restricted Lie algebra. A $U_0(\fg)$-module $M$ such that $\fg.M=(0)$ will be referred to as being 
{\it trivial}. 

\bigskip

\begin{Lemma} \label{TrLA0} Let $\fg$ be a $p$-trivial restricted Lie algebra of dimension $\dim_k\fg \!\ge\! 2$. If $M \in \EIP(\fg)\cap\EKP(\fg)$, then $M$ is trivial. \end{Lemma}

\begin{proof} We assume $M\ne (0)$, so that $M$ has constant Jordan type
\[\Jt(M) = \bigoplus_{i=1}^d a_i[i],\]
with $1\!\le\!d\!\le\!p$ and $a_d \ne 0$. The nilpotent Lie algebra $\fg$ has a nontrivial center $C(\fg)$. As $\dim_k\fg\!\ge\!2$, we can thus find $\fe \in \EE(2,\fg)$. Since $U_0(\fe) \cong k(\ZZ/(p)\!\times\!
\ZZ/(p))$, the results of \cite{CFS} may be applied to the $U_0(\fe)$-module $N:=M|_\fe$, which belongs to $\EIP(\fe)\cap\EKP(\fe)$ and has constant Jordan type $\Jt(N)=\Jt(M)$.

Suppose that $d\!\ge\!2$. Being a submodule of an equal kernels module, it follows that $N':=\Rad^{d-2}(N) \in \EIP(\fe)$ is an equal kernels module, cf.\ \cite[(1.9)]{CFS}. According to
\cite[(4.1)]{CFS}, there exists a decomposition
\[ N' \cong \bigoplus_{i=1}^\ell W_{n_i,2},\]
with all constituents having the equal kernels property. Since $d\!\ge\!2$ and $\Jt(W_{n_i,2})=(n_i\!-\!1)[2]\!\oplus\![1]$, there is $i \in \{1,\ldots,\ell\}$ such that $n_i\ge 2$. Writing $U_0(\fe)=k[x,y]$ with 
$x^p=0=y^p$, we observe $\ker x_{W_{n_i,2}} \ne \ker y_{W_{n_i,2}}$, a contradiction. Thus, $d=1$, whence $\Jt(M)=(\dim_kM)[1]$. As a result, $M$ is a trivial $U_0(\fg)$-module. \end{proof}

\bigskip
\noindent
To illustrate our geometric methods, we begin with the following elementary observation. Suppose that $\fg$ is an $r$-dimensional $p$-trivial restricted Lie algebra, and let $M$ be a non-trivial $U_0(\fg)
$-module of constant rank. If $\varrho_M : U_0(\fg) \lra \End_k(M)$ is the representation afforded by $M$, then $M$ being non-trivial implies $\rk(M)\ne 0$, so that $\varrho_M|_{\fg} : \fg \lra \End_k(M)$ 
is injective. Engel's Theorem guarantees that $\im \varrho_M$ may be embedded into the space of strictly upper triangular matrices of size $\dim_kM$, so that $r \le \binom{\dim_kM}{2}$. Hence every 
$U_0(\fg)$-module of constant rank of dimension $\dim_kM\!\le \sqrt{2r}\!+\!\frac{1}{2}$ is trivial. Tango's Theorem allows us to relax this condition. 

\bigskip

\begin{Theorem} \label{TrLA1} Suppose that $\fg$ is $p$-trivial and let $M$ be a $U_0(\fg)$-module of constant rank.
\begin{enumerate}
\item If $\dim_k\Rad(M) \le \dim_k\fg\!-\!1$, then $M$ has the equal images property.
\item If $\dim_kM/\Soc(M) \le \dim_k\fg\!-\!1$, then $M$ has the equal kernels property.
\item If $\max\{\dim_k\Rad(M),\dim_kM/\Soc(M)\} \le \dim_k\fg\!-\!1$, then $M$ is trivial. 
\item If $\dim_kM \le \dim_k\fg$, then $M$ is trivial. \end{enumerate}\end{Theorem}

\begin{proof} We put $r:= \dim_k\fg$. Since $\fg$ is $p$-trivial, the algebra $U_0(\fg)$ is local and $\fg \subseteq \Rad(U_0(\fg))$. It follows that $\im x_M \subseteq \Rad(M)$ for every
$x \in \fg$.

(1) Suppose that $\dim_k\Rad(M) \le r\!-\!1$ and put $d:= \rk(M)$. Corollary \ref{RLA1} ensures that
\[ \msim_M : \PP^{r-1} \lra \Gr_d(M) \ \ ; \ \ [x] \mapsto \im x_M\]
is a morphism of projective varieties. Since $\msim_M(x) \subseteq \Rad(M)$ for all $x \in \PP^{r-1}$, the morphism $\msim_M$ factors through $\Gr_d(\Rad(M)) \subseteq \Gr_d(M)$. 
As $\dim_k\Rad(M)\le r\!-\!1$, we may invoke Theorem \ref{Ta1}(1) to see that the map $\msim_M$ is constant. Corollary \ref{LASN3} now shows that the $U_0(\fg)$-module $M$ has the equal images 
property. 

(2) We consider the dual $U_0(\fg)$-module $M^\ast$, so that $x_{M^\ast}=-(x_M)^{\tr}$ for all $x \in \fg$. Consequently, 
\[ (\ast) \ \ \ \ \ \ker x_{M^\ast} = \{\lambda \in M^\ast \ ; \ \lambda\circ x_M = 0\},\]
showing that $M^\ast$ also has constant rank $\rk(M^\ast)=\rk(M)$. 

Recall the canonical pairing $M^\ast \times M \lra k \ \ ; \ \ (f,m) \mapsto f(m)$. Since $(x.f)(\Soc(M)) =(0)$ for every $x \in \Rad(U_0(\fg))$, we have $\Rad(M^\ast) \subseteq {}^\perp\Soc(M)$. Let $f \in 
{}^\perp\Soc(M)$. If $U \subseteq M^\ast$ is a maximal submodule, then $U^\perp \subseteq M$ is simple, so that $f(U^\perp)=(0)$ and $f \in {}^\perp (U^\perp)= U$. This shows that $f \in \Rad(M^\ast)$. 
Consequently, $\Rad(M^\ast) = {}^\perp\Soc(M) \cong (M/\Soc(M))^\ast$.      

Since $\dim_k\Rad(M^\ast)=\dim_kM/\Soc(M) \le r\!-\!1$, part (1) implies that $M^\ast$ has the equal images property. In view of ($\ast$), the module $M$ has the equal kernels property. 

(3) Owing to (1) and (2), the $U_0(\fg)$-module $M$ belongs to $\EIP(\fg)\cap \EKP(\fg)$. Now Lemma \ref{TrLA0} forces $M$ to be trivial. 

(4) Since $\dim_kM \le r$, the condition of (3) is fulfilled. \end{proof}

\bigskip

\begin{Example} The bound of Theorem \ref{TrLA1}(4) cannot be improved: Let $V_{r+1} = \bigoplus_{i=1}^{r+1}kv_i$ be the $(r\!+\!1)$-dimensional vector space on which the elementary Lie 
algebra $\fe_r := \bigoplus_{i=1}^rkx_i$ acts via
\[ x_i\dact v_j = \delta_{i,j} v_{r+1} \ \ \ \text{for} \ 1\le i\le r \ \text{and} \ 1\le j\le r\!+\!1.\]
This $U_0(\fe_r)$-module has constant rank $1$. Note that $V_{r+1}$ is isomorphic to $\Soc_2(U_0(\fe_r))$.\end{Example} 

\bigskip
\noindent
For non-abelian $p$-trivial Lie algebras we have the following sharpening:

\bigskip  

\begin{Corollary} \label{TrLA2} Let $\fg$ be a non-abelian $p$-trivial Lie algebra. If $M$ is a non-trivial $U_0(\fg)$-module of constant rank, then $\min\{\dim_k\Rad(M),\dim_kM/\Soc(M)\} \ge \dim_k\fg$. 
\end{Corollary}

\begin{proof} If $\dim_k\Rad(M) \le \dim_k\fg\!-\!1$, then Theorem \ref{TrLA1}(1) and Corollary \ref{LASN3} imply that $M$ is trivial, a contradiction. The assumption $\dim_kM/\Soc(M) \le \dim_k\fg\!-\!1$ 
yields the same conclusion for $M^\ast$ and hence also for $M$. \end{proof}

\bigskip

\begin{Corollary} \label{TrLA3} Let $\fg$ be $p$-trivial. If $M$ is a $U_0(\fg)$-module of constant rank and of dimension $\dim_k\fg\!+\!1$, then one of the following cases occurs:
\begin{enumerate}
\item[(a)] $M \cong k^{\dim_k\fg+1}$, or
\item[(b)] $\fg$ is abelian and $M \cong \Soc_2(U_0(\fg)), \, U_0(\fg)/\Rad^2(U_0(\fg))$.\end{enumerate} \end{Corollary}

\begin{proof} Let $r:= \dim_k\fg$, so that $\dim_k\Rad(M)\le r$. If $\rk(M)=0$, then $M \cong k^{r+1}$. We therefore assume that $\rk(M)\ne 0$.

Suppose that $\dim_k \Rad(M) \le r\!-\!1$. By Corollary \ref{TrLA2}, $\fg$ is abelian and hence isomorphic to $\fe_r$. Moreover, Theorem \ref{TrLA1} shows that $M \in \EIP(\fe_r)$. Thanks to \cite[(1.9)]
{CFS}, we have $\Rad(M) \in \EIP(\fe_r)$ and Theorem \ref{TrLA1} implies $\Rad^2(M)=(0)$, whence $\ell\ell(M)=2$. Since $M$ is not trivial, Theorem \ref{TrLA1} yields $\dim_k\Soc(M)=1$, so that 
there exists an embedding $\iota : M \hookrightarrow U_0(\fe_r)$. As $\ell\ell(M) = 2$, this map factors through $\Soc_2(U_0(\fe_r))$. Thus, $\im \iota = \Soc_2(U(\fe_r))$ for dimension reasons.

Since $M^\ast$ has constant rank $\rk(M^\ast)=\rk(M)$, the assumption $\dim_kM/\Soc(M) \le r\!-\!1$ implies $M \cong \Soc_2(U_0(\fe_r))^\ast \cong U_0(\fe_r)/\Rad^2(U_0(\fe_r))$. 

It thus remains to consider the case, where $\dim_k\Rad(M)=r$ and $\dim_k\Soc(M)=1$.  As $M$ has constant rank $\rk(M)\ne 0$, this implies $\Soc(M) \subseteq \im x_M$ for
all $x \in \fg\!\smallsetminus\!\{0\}$.  Consequently, the factor module $M/\Soc(M)$ has constant rank. By Theorem \ref{TrLA1}, this module is trivial, whence $\Rad(M) \subseteq \Soc(M)$. As a result, 
$r=1$, so that $\fg \cong \fe_1$ and $M \cong \Soc_2(U_0(\fe_1))$. \end{proof} 

\bigskip
\noindent
Let $\fg$ be $p$-trivial. By the foregoing result, $U_0(\fg)$-modules of constant rank of dimension $\dim_k\fg\!+\!1$ belong to $\EIP(\fg)\cup\EKP(\fg)$. The example of Section \ref{S:RDF} shows that 
for modules of dimension $\dim_k\fg\!+\!2$, this may not be the case. 

\bigskip 

\begin{Corollary} \label{MSD4} Suppose that $\fg$ is $p$-trivial and let $M$ be a $U_0(\fg)$-module of constant rank such that $[i]\oplus n[1] \in \Jt(M)$ for some $n\!\ge\! 1$ and $2\! \le\! i\! \le\! p$. Then 
we have $n\! \ge\! \dim_k\fg\!-\!i\!+\!1$. \end{Corollary} 

\begin{proof} Since $i\! \ge\! 2$, the $U_0(\fg)$-module $M$ is not trivial. Consequently, Theorem \ref{TrLA1} implies 
\[\dim_k\fg \!+\!1 \le \dim_kM = n\!+\!i,\]
as asserted.  \end{proof} 

\bigskip

\begin{Definition} Let $(\fg,[p])$ be a restricted Lie algebra. Then
\[ \prk(\fg) := \max\{\dim_k\fh \ ; \ \fh \subseteq \fg \ \text{$p$-trivial subalgebra}\}\]
is called the {\it $p$-trivial rank of $\fg$}. \end{Definition}

\bigskip
\noindent
Since $p$-trivial Lie algebras of dimension $>0$ have non-trivial centers, we see that $\EE(2,\fg) \ne \emptyset$ if and only if $\prk(\fg) \ge 2$.

Let $G$ be a reductive algebraic group. We denote by $h_G$ the {\it Coxeter number} of $G$. The dimension of any maximal torus $T \subseteq G$ is called the {\it rank} $\rk(G)$ of $G$.

\bigskip

\begin{Proposition} \label{TrLA4} Let $(\fg,[p])$ be a restricted Lie algebra, $M$ be a $U_0(\fg)$-module of constant rank.
\begin{enumerate}
\item If $\dim_kM\! \le\! \prk(\fg)$, then $\rk(M)=0$.
\item Suppose that $\fg = \Lie(G)$, where $G$ is reductive such that $p\ge h_G$. If $\dim_kM \le \frac{1}{2}(\dim G\!-\! \rk(G))$, then $M$ is a direct sum of one-dimensional modules. \end{enumerate} \end{Proposition}

\begin{proof} (1) Let $\fh \subseteq \fg$ be a $p$-trivial subalgebra of dimension $\prk(\fg)$. Then $N:= M|_{\fh}$ is a $U_0(\fh)$-module of constant rank $\rk(N)=\rk(M)$. In view of Theorem \ref{TrLA1}(4), the module $N$ is trivial so that $\rk(N)=0$.

(2) Since $p\ge h_G$, the unipotent radical of a Borel subalgebra of $\fg$ is $p$-trivial, whence $\dim_kM\le \prk(\fg)$. Let $T \subseteq G$ be a maximal torus with root system $\Phi$. In view of (1), the 
elements of $\bigcup_{\alpha \in \Phi}\fg_\alpha \subseteq V(\fg)$ act trivially on $M$. Hence $\bigoplus_{\alpha \in \Phi}\fg_\alpha$ acts trivially on $M$, and our assertion follows from the decomposition 
$\fg = \Lie(T)\!\oplus\!\bigoplus_{\alpha \in \Phi}\fg_\alpha$. \end{proof} 

\bigskip

\section{Modules for finite group schemes} \label{S:FG}
We now turn to the general case concerning modules over a finite group scheme $\cG$. This requires the Friedlander-Pevtsova theory of $p$-points, set forth
in a series of articles, beginning with \cite{FPe05}. Let $\fA_p := k[T]/(T^p)$ be the truncated polynomial ring with canonical generator $t:= T\!+\!(T^p)$. For an algebra homomorphism $\alpha : 
\fA_p \lra k\cG$ we denote by $\alpha^\ast : \modd k\cG \lra \modd \fA_p$ the associated pull-back functor. We say that $\alpha$ is a {\it $p$-point}, provided
\begin{enumerate}
\item[(P1)] $\alpha$ is left flat, i.e.\ $\alpha^\ast(k\cG)$ is projective, and
\item[(P2)] there exists an abelian unipotent subgroup scheme $\cU \subseteq \cG$ such that $\im \alpha \subseteq k\cU$. \end{enumerate} 
The set of $p$-points of $\cG$ will be denoted $\Pt(\cG)$. Two $p$-points $\alpha, \beta$ are said to be {\it equivalent} ($\alpha\!\sim\!\beta$) if we have 
\[ \alpha^\ast(M) \ \text{is projective} \ \Leftrightarrow \ \beta^\ast(M) \ \text{is projective} \]
for every $M \in \modd \cG$. By results of \cite{FPe05}, the space $\Pp(\cG):= \Pt(\cG)/\!\!\sim$ of equivalence classes of $p$-points is a noetherian topological space. 

If $\cH \subseteq \cG$ is a subgroup of the finite algebraic group $\cG$, then the canonical inclusion induces a continuous map
\[ \iota_\ast : \Pp(\cH) \lra \Pp(\cG)\]
which usually is not injective.

\bigskip

\subsection{Modules defined via $p$-points}
Using $p$-points one can extend the concepts of constant rank modules and equal images modules to $\cG$-modules, cf.\ \cite{FPe10}. Let $M$ be a $\cG$-module. Given $j \in \{1,\ldots,p\!-\!1\}$, 
we let
\[ \rk^j(M) := \max\{ \rk(\alpha(t)^j_M) \ ; \ \alpha \in \Pt(\cG)\}\]
be the {\it generic $j$-rank} of $M$. We say that $M \in \modd \cG$ has {\it constant $j$-rank} if $\rk(\alpha(t)^j_M) = \rk^j(M)$ for all $\alpha \in \Pt(\cG)$. Modules of constant $1$-rank are
referred to as being of {\it constant rank}. The $\cG$-module $M$ has the {\it equal images property}, if, for every $\ell \in \{1,\ldots,p\!-\!1\}$, there is a subspace $V_\ell \subseteq M$
such that $\im \alpha(t)^\ell_M=V_\ell$ for all $\alpha \in \Pt(\cG)$. When dealing with infinitesimal groups of height $r$ we shall often identify $\fA_p$ with the subalgebra $k[u_{r-1}]$ of $k\GG_{a(r)}$.

The following result shows that in the context of infinitesimal groups our new definitions are compatible with the previous ones.

\bigskip

\begin{Lem} \label{FG1} Let $\cG$ be an infinitesimal group scheme of height $r$, $M$ be a $\cG$-module.
\begin{enumerate}
\item Let $j \in \{1,\ldots, p\!-\!1\}$. If $\rk(\alpha(u_{r-1})^j_M) = \rk(\beta(u_{r-1})^j_M)$ for all $\alpha,\beta \in V(\cG)\!\smallsetminus\!\{\varepsilon\}$, then $M$ has constant $j$-rank.
\item Suppose there exist subspaces $V_1,\ldots,V_{p-1} \subseteq M$ such that $\im \alpha(u_{r-1})^\ell_M = V_\ell$ for all  $\alpha \in V(\cG)\!\smallsetminus\!\{\varepsilon\}$ and
$\ell \in \{1,\ldots,p\!-\!1\}$. Then $M$ has the equal images property. \end{enumerate} \end{Lem}

\begin{proof} In view of \cite[(3.8)]{FPe05}, the map
\[ (\ast) \ \ \ \ \ \ \ \Xi_\cG : \Proj(V(\cG)) \lra \Pp(\cG) \ \ ; \ \ [\alpha] \mapsto [\alpha|_{k[u_{r-1}]}] \]
is bijective. 

(1) This is a direct consequence of \cite[(3.8)]{FPe10} and ($\ast$).

(2)  We first assume that $\cG=\cU$ is an infinitesimal abelian unipotent group of height $r$. General theory \cite[(14.4)]{Wa} provides an isomorphism
\[ k\cU \cong k[X_1,\ldots, X_s]/(X_1^{p^{n_1}}, \ldots, X_s^{p^{n_s}}) \ \ \ \ ; \ \ \ \ n_i \in \NN.\]
We write $v_i := X_i\!+\!(X_1^{p^{n_1}}, \ldots, X_s^{p^{n_s}})$ and put $kE := k[v_1^{p^{n_1-1}},\ldots,v_s^{p^{n_s-1}}]$. Hence $kE$ looks like the group algebra of a $p$-elementary
abelian group of rank $s$. According to \cite[(1.6)]{Fa07}, there exists a $kE$-linear projection $\pr_E : k\cU \lra kE$ with kernel $\ker \pr_E = \Rad(kU)kE$ that  induces a bijection
\[ \pr_{E,\ast} : \Pp(\cU) \lra \Pp(E) \ \ ;  \ \ [\alpha] \mapsto [\alpha_{(E)}],\]
where $\alpha_{(E)}$ is the unique $p$-point such that $\alpha_{(E)}(u_{r-1}):= \pr_E(\alpha(u_{r-1}))$. It now follows from ($\ast$) and \cite[(2.2)]{FPS} in conjunction with \cite[(1.4)]{Fa07} that ther
exist $\alpha_1,\ldots,\alpha_s \in V(\cU)$ and $\lambda_1,\ldots,\lambda_s \in k^\times$ such that 
\[ v_i^{p^{n_i-1}} \equiv \lambda_i\alpha_i(u_{r-1}) \ \ \ \ \ \modd \Rad(kU)\Rad(kE).\]
In particular, $kU\Rad(kE) = (v_1^{p^{n_1-1}},\ldots,v_s^{p^{n_s-1}}) = (\alpha_1(u_{r-1}),\ldots, \alpha_s(u_{r-1}))$, so that
\[ \im (v_i)^{p^{n_i-1}}_M \subseteq V_1 \ \ \ \ \ \ \ \ \text{for all} \ i \in \{1,\ldots, s\}.\]
Now let $\alpha \in \Pt(\cU)$ be a $p$-point. Since $\alpha(u_{r-1})^p=0$ there exist $f_i \in k\cU$ such that
\[ \alpha(u_{r-1}) = \sum_{i=1}^s v_i^{p^{n_i-1}}f_i.\]
As a result,
\[ \im \alpha(u_{r-1})_M \subseteq \sum_{i=1}^s \im (v_i)^{p^{n_i-1}}_M \subseteq V_1.\]
By virtue of (1), the $\cU$-module $M$ has constant rank $\rk(M)=\dim_kV_1$, so that $\dim_k \alpha(u_{r-1})_M=V_1$. We conclude that $\im \alpha(u_{r-1})_M = V_1$.

Since $k\cU$ is abelian, it we readily obtain $\im \alpha(u_{r-1})^\ell_M=\im \beta(u_{r-1})^\ell_M$ for all $\alpha, \beta \in \Pt(\cU)$ and $\ell \in \{1,\ldots,p\!-\!1\}$. As a result, 
the module $M$ has the equal images property.

In the general case, we let $\alpha \in \Pt(\cG)$ be a $p$-point. By definition, there exists an abelian unipotent subgroup scheme $\cU \subseteq \cG$ such that $\im \alpha
\subseteq k\cU$. Since $V(\cU) \subseteq V(\cG)$ (cf.\ \cite[(1.5)]{SFB1}), the first part of the proof implies that $\im \alpha(u_{r-1})_M^\ell = V_\ell$ for every $\ell \in \{1,\ldots,p\!-\!1\}$. This shows that
the $\cG$-module $M$ has the equal images property. \end{proof}

\bigskip
\noindent
For an abelian unipotent group scheme $\cU$, condition (P2) above is automatic, so that $p$-points $\alpha : \fA_p \lra k\cU$ are flat algebra homomorphisms, a requirement that makes no reference
to the coalgebra structure of $k\cU$. As noted above, we have an isomorphism
\[ k\cU \cong k[X_1,\ldots, X_s]/(X_1^{p^{n_1}}, \ldots, X_s^{p^{n_s}}) \ \ \ \ ; \ \ \ \ n_i \in \NN\]
of associative algebras. A truncated polynomial ring $k[X_1,\ldots,X_s]/(X_1^{p^{n_1}}, \ldots, X_s^{p^{n_s}})$ can be interpreted as the restricted enveloping algebra $U_0(\fn)$ of the abelian restricted 
Lie algebra $\fn = \bigoplus_{i=1}^s \fn_{n_i}$, where $\fn_{n_i} := \bigoplus_{j=0}^{n_i-1}kx_i^{[p]^j} \ \ ; \ \ x_i^{[p]^{n_i-1}} \ne 0 = x_i^{[p]^{n_i}}$ is the $n_i$-dimensional nil-cyclic restricted Lie 
algebra and $s=\dim V(\fn)$. It follows from \cite[(3.8)]{FPe05} that $s = \dim \Pp(\cU)\!+\!1$. We shall exploit this observation to generalize some of our earlier results.

\bigskip

\begin{Thm} \label{FG2} Let $\cU$ be an abelian unipotent group scheme. Suppose that $M$ is a $\cU$-module of constant rank. Then the following 
statements hold:
\begin{enumerate}
\item If $\dim_k\Rad(M) \le \dim \Pp(\cU)$, then $M$ has the equal images property.
\item If $\max\{\dim_k\Rad(M),\dim_kM/\Soc(M)\} \le \dim \Pp(\cU)$, then $\rk(M)=0$. 
\item If $\dim_kM \le \dim \Pp(\cU)\!+\!1$, then $\rk(M)=0$. \end{enumerate}\end{Thm}

\begin{proof} (1) Observing \cite[(14.4)]{Wa}, we write $k\cU \cong  k[X_1,\ldots,X_s]/(X_1^{p^{n_1}}, \ldots, X_s^{p^{n_s}}) \cong U_0(\fn)$, so that the spaces of flat points coincide. As $\fn$ is 
abelian, the nullcone $V(\fn)$ is the elementary restricted Lie algebra $\fe_s$ of dimension $s=\dim V(\fn)=\dim \Pp(\cU)\!+\!1$. Moreover, since $V(\fn)=V(\fe_s)=\fe_s$, an application of \cite[(3.8)]
{FPe05} shows that the canonical inclusion $\iota : \fe_s \hookrightarrow \fn$ induces a homeomorphism 
\[ \iota_\ast : \Pp(\fe_s) \lra \Pp(\fn) \ \ ; \ \ [\alpha] \mapsto [\iota\circ \alpha].\]
If $\fn = \bigoplus_{i=1}^s\fn_{n_i}$, then $\fe_s = \bigoplus_{i=1}^s\fn_{n_i}^{[p]^{n_i-1}}$, whence $\Rad(M|_{\fe_s}) \subseteq \Rad(M)|_{\fe_s}$. By Theorem \ref{TrLA1}(1), the map $\msim_M : 
\PP(\fe_s) \lra \Gr_{\rk(M)}(M)$ is constant. Since $\fe_s$ is abelian, a consecutive application of Theorem \ref{LASN2} and Lemma \ref{FG1} yields the assertion.

(2) Since $\Soc(M)|_{\fe_s} \subseteq \Soc(M|_{\fe_s})$, it follows from Theorem \ref{TrLA1} that $M|_{\fe_s}$ is trivial. Hence Lemma \ref{FG1} implies $\rk(M)=\max \{\rk(x_M) \ ; \ x \in \fe_s\} =0$. 

(3) This is a direct consequence of (2). \end{proof}

\bigskip
\noindent
In view of $\dim \Pp(\GG_{a(r)})\!+\!1 = r$, the foregoing result shows in particular, that a $\GG_{a(r)}$-module $M$ of constant rank with $\dim_kM \le r$ is trivial. This strengthens \cite[(3.18)]{FPe11}.

\bigskip
\noindent
Given a finite group scheme $\cG$, we let 
\[ \aurk(\cG) := \max \{\dim \Pp(\cU)\!+\!1 \ ; \ \cU \subseteq \cG \ \text{abelian unipotent subgroup}\}\]
be the {\it abelian unipotent rank} of $\cG$. If $G$ is a finite group, Quillen's Dimension Theorem \cite[(5.3.8)]{Be2} ensures that this number coincides with the {\it $p$-rank} $\rk_p(G)$ of $G$, that is, 
the maximum of all ranks of the $p$-elementary abelian subgroups of $G$. Thus, for finite groups $G$, we have $\aurk(G) = \dim \Pp(G)\!+\!1$. In general, there is an inequality $\aurk(\cG) \le 
\dim \Pp(\cG)\!+\!1$, with both numbers possibly being arbitrarily far apart: Assuming $p\ge 3$, we let $\fh_n$ be the $(2n\!+\!1)$-dimensional $p$-trivial Heisenberg algebra. Then $\aurk(\fh_n) = n\!+\!
1$, while $\dim \Pp(\fh_n)\!+\!1= 2n\!+\!1$.

In Section \ref{S:EA} below we will see that the rank of a group scheme is computable via elementary abelian subgroup schemes. 

\bigskip

\begin{Cor} \label{FG3} Let $M$ be a $\cG$-module of constant rank. If $\dim_kM \le \aurk(\cG)$, then $\rk(M)=0$.\end{Cor}

\begin{proof} Let $\cU \subseteq \cG$ be an abelian unipotent subgroup such that $\dim \Pp(\cU)\!+\!1=\aurk(\cG)$. Since $M|_\cU$ has constant rank, Theorem \ref{FG2} implies that $\rk(M)= \rk(M|_
\cU)=0$. \end{proof}

\bigskip
\noindent
Let $\cG$ be a finite group scheme. We denote by $\cx_\cG(k)$ the {\it complexity} of the trivial $\cG$-module, cf.\ \cite[(\S 5.1)]{Be2}. Thanks to \cite[(5.6)]{FPe05}, we have $\cx_\cG(k)=\dim 
\Pp(\cG)\!+\!1$.

Since many of our results will require the assumption $\aurk(\cG)\ge 2$, we indicate a structural ramification of this condition:

\bigskip 

\begin{Lem} \label{FG4} Let $\cG$ be a finite group scheme such that $\aurk(\cG) \ge 2$. Then there exists a subgroup scheme $\cE \subseteq \cG$ such that $\cE$ is isomorphic 
to one of the following group schemes: $\GG_{a(2)}, \GG_{a(1)}\!\times\!\GG_{a(1)}, \GG_{a(1)}\!\times\!E_1, E_2$. \end{Lem}

\begin{proof} By assumption, there exists an abelian unipotent subgroup $\cU \subseteq \cG$ such that $\dim \Pp(\cU) \ge 1$. General theory provides a decomposition
\[ \cU = \cU^0\!\times\!\cU_{\rm red}.\]
If both factors are non-trivial, then $\GG_{a(1)}\subseteq \cU^0$ and $E_1 \subseteq \cU_{\rm red}$, so that $\cE := \GG_{a(1)}\!\times\!E_1$ is the desired group. If $\cU^0 = e_k$,
then Quillen's dimension theorem provides an elementary abelian subgroup $E_r \subseteq \cU(k)$ such that $r=\dim \Pp(\cU)\!+\!1$. Thus, $\cE:=E_2 \subseteq \cU_{\rm red}$ is a suitable
subgroup. In the remaining case, $\cU$ is an infinitesimal unipotent subgroup. Let $\fu := \Lie(\cU)$ be its Lie algebra, so that $V(\fu)$ is an elementary subalgebra. If $\dim V(\fu) \ge 2$, $\fe_2
\subseteq V(\fu)$. It follows that $\cU'$ contains a subgroup $\cE$ that is isomorphic to $\GG_{a(1)}\!\times\!\GG_{a(1)}$. Alternatively, $\dim V(\fu)=1$, so that $\cU$ contains exactly one subgroup of 
type $\GG_{a(1)}$. If it contains no subgroups of type $\GG_{a(2)}$, then \cite[(5.2)]{FRV} yields $1=\cx_{\cU}(k)=1\!+\!\dim \Pp(\cU)$, a contradiction. \end{proof}

\bigskip

\subsection{Modules for elementary abelian group schemes}\label{S:EA}
In this section, we introduce a class of group schemes that are natural generalizations of $p$-elementary abelian groups and elementary restricted Lie algebras. Recall that a finite reduced group 
scheme $\cG$ is completely determined by its finite group $\cG(k)$ of $k$-rational points. Moreover, any finite group $G$ gives rise to a reduced group scheme $\cG_G={\rm Spec}_k(kG^\ast)$, 
where $\cG_G(k)=G$. We shall henceforth not distinguish between a finite group $G$ and its associated reduced group scheme $\cG_G$. 

\bigskip

\begin{Definition} An abelian group scheme $\cE$ is called {\it elementary abelian}, provided there exist subgroups
$\cE_1,\ldots,$ $\cE_n \subseteq \cE$ such that
\begin{enumerate}
\item[(a)] $\cE=\cE_1\cdots\cE_n$, and
\item[(b)] for each $i \in \{1,\ldots,n\}$, we have  isomorphisms $\cE_i \cong \GG_{a(r_i)}$ or $\cE_i \cong E_1$.
\end{enumerate}  \end{Definition}

\bigskip
\noindent
Since the group schemes $E_1$ and $\GG_{a(1)}$ are simple, it follows that a reduced elementary abelian group scheme $\cE$ is isomorphic to some $E_r$, while we have $\cE\cong (\GG_{a(1)})^r$ 
for every infinitesimal elementary abelian group scheme $\cE$ of height $1$. In the latter case, there are isomorphisms $\Lie(\cE)\cong \fe_r$ and $k\cE \cong U_0(\fe_r)$.

The dimension $\dim_kk\cG$ of a finite group scheme $\cG$ is also referred to as the {\it order} of $\cG$. Thus, an abelian unipotent subgroup $\cU\subseteq \cG$ is contained in an abelian unipotent 
subgroup $\cU_0$, whose order is maximal subject to these properties. The group scheme $\cU_0$ is maximal subject to being abelian and unipotent. 

Our next result shows that Hopf algebras of elementary abelian group schemes are isomorphic (as associative algebras) to group algebras of $p$-elementary abelian groups.

\bigskip

\begin{Lem} \label{EA1} Let $\cU$ be a finite abelian unipotent group scheme. Then the following statements hold:
\begin{enumerate}
\item There exists a unique elementary abelian subgroup $\cE_\cU \subseteq \cU$ that contains any other elementary abelian subgroup of $\cU$. 
\item The canonical map $\iota_\ast : \Pp(\cE_\cU) \lra \Pp(\cU)$ is a homeomorphism.
\item If $\cU$ is elementary abelian, then $k\cU \cong k[X_1,\ldots, X_n]/(X_1^p,\ldots,X^p_n)$, where $n=\aurk(\cU)$. \end{enumerate} \end{Lem}

\begin{proof} (1) We let $\cE_{\cU}$ be an elementary abelian subgroup of maximal order. If $\cE \subseteq \cU$ is elementary abelian, then $\cE\cE_{\cU}$ is an elementary
abelian subgroup of $\cU$ containing $\cE_{\cU}$. Consequently, $\cE \subseteq \cE\cE_{\cU}=\cE_{\cU}$, as desired. 

(2) We denote by $\iota : \cE_\cU \lra \cU$ the canonical inclusion. Let $\alpha : \fA_p \lra k\cU$ be a $p$-point. Thanks to \cite[(4.2)]{FPe05}, there exists a $p$-point $\beta \in \Pt(\cU)$ and an elementary abelian subgroup $\cE \subseteq \cU$
such that $\alpha\! \sim\! \beta$ and $\im \beta \subseteq k\cE$. In view of (1), this implies that $[\alpha] \in \im \iota_\ast$.

Now suppose that $\alpha, \beta \in \Pt(\cE_\cU)$ are $p$-points such that $\iota_\ast([\alpha])=\iota_\ast([\beta])$. Let $M \in \modd \cE_\cU$. Since the abelian algebra $k\cU$ is
free over $k\cE_\cU$, we have $(k\cU\!\otimes_{k\cE_\cU}\!M)|_{k\cE_\cU}\cong M^n$, where $n$ is the rank of $k\cU$ over $k\cE_\cU$. Thus, if $\alpha^\ast(M)$ is projective, then
$(\iota\!\circ\!\alpha)^\ast(k\cU\!\otimes_{k\cE_\cU}\!M) \cong \alpha^\ast(M)^n$ is projective. Since $\iota_\ast([\alpha])=\iota_\ast([\beta])$, it follows that $\beta^\ast(M)^n \cong 
(\iota\!\circ\!\beta)^\ast(k\cU\!\otimes_{k\cE_\cU}\!M)$ is also projective. Consequently, the module $\beta^\ast(M)$ is projective. As a result, $\alpha\!\sim\!\beta$, so that $\iota_\ast$ is injective. 
The same arguments show that $\iota_\ast(\Pp(\cE_\cU)_M)=\Pp(\cU)_{k\cU\!\otimes_{k\cE_\cU}\!M}$ for all $M \in \modd \cE_\cU$. Hence the continuous bijective map $\iota_\ast$ is closed 
and therefore a homeomorphism.

(3) By assumption, there exists a quotient map $\prod_{i=1}^n\cE_i \lra \cU$, where $\cE_i = \GG_{a(r_i)},E_1$. There results a surjection
\[ \gamma : \bigotimes_{i=1}^n k\cE_i \twoheadrightarrow k\cU,\]
of Hopf algebras. As both algebras are local, we have $\gamma(\Rad(\bigotimes_{i=1}^n k\cE_i))=\Rad(k\cU)$. Since $x^p=0$ for all $x \in \Rad(\bigotimes_{i=1}^n k\cE_i)$, we conlude that $x^p=0$ for 
all $x \in \Rad(k\cU)$.

As $\cU$ is abelian an unipotent, general theory (\cite[(14.4)]{Wa}) provides an isomorphism
\[ k\cU \cong k[X_1,\ldots,X_n]/(X_1^{p^{a_1}},\ldots,X_n^{p^{a_n}}),\]
where $a_i \in \NN$ and $\aurk(\cU)=\dim \Pp(\cU)\!+\!1=\cx_\cU(k) =n$ coincides with the complexity $\cx_{\cU}(k)$ of the trivial $\cU$-module (cf.\ \cite[(5.6)]{FPe05}). By the above, we have 
$X_i^p \in (X_1^{p^{a_1}},\ldots,X_n^{p^{a_n}})$. By applying the canonical map $\omega_i : k[X_1,\ldots,X_n] \lra k[X_i]$ sending $X_j$ onto $\delta_{ij}X_i$, we see that $X_i^p \in (X_i^{p^{a_i}})$. This implies $a_i=1$. \end{proof}

\bigskip

\begin{Example} If $U=\cU$ is reduced, then $U$ is an abelian $p$-group and $\cE_\cU$ is the subgroup of elements of order $\le p$. \end{Example}

\bigskip

\begin{Cor} \label{EA2} Let $\cE$ be a finite group scheme.
\begin{enumerate} 
\item If $k\cE\cong k[X_1,\ldots,X_n]/(X_1^p,\ldots,X_n^p)$, then $\cE$ is elementary abelian and $n=\aurk(\cE)$.
\item If $\cE$ is elementary abelian and $\cE'\subseteq \cE$ is a subgroup, then $\cE'$ is elementary abelian. \end{enumerate}\end{Cor}

\begin{proof} (1) Since $k\cU$ is abelian and local, the group scheme $\cU$ is abelian unipotent and $n=\aurk(\cU)=\dim\Pp(U)\!+\!1$. In view of Lemma \ref{EA1}(2), the map 
$\iota_\ast : \Pp(\cE_\cU)\lra \Pp(\cU)$ is a homeomorphism, so that $\dim \Pp(\cE_{\cU})=\dim\Pp(\cU)$. We thus conclude $\aurk(\cE_\cU)=\aurk(\cU)$, so that Lemma \ref{EA1}(3) 
together with our current assumption implies $\dim_kk\cE_\cU = p^n = \dim_kk\cU$. As a result, the group $\cU=\cE_\cU$ is elementary abelian.

(2) Since $\cE$ is elementary abelian, Lemma \ref{EA1} implies that $x^p =0$ for all $x \in \Rad(k\cE)$. By general theory, there is an isomorphism
\[ k\cE' \cong k[X_1,\ldots,X_r]/(X_1^{p^{n_1}},\ldots,X_r^{p^{n_r}}),\]
where $n_i \ge 1$. As $k\cE$ is abelian, we have $\Rad(k\cE')\subseteq\Rad(k\cE)$, so that $n_i=1$. Part (1) now shows that $\cE'$ is elementary abelian. \end{proof}

\bigskip
\noindent
By general theory, an elementary abelian group scheme $\cE$ is the direct product $\cE=\cE^0\!\times\!\cE_{\rm red}$ of its infinitesimal and reduced parts. By the above, we have
$\cE_{\rm red}\cong E_r$ for some $r\ge 0$.

We turn to the definition of $j$-degrees for modules of constant $j$-rank over an elementary abelian group $\cE$ with $\aurk(\cE)=n$. In view of Lemma \ref{EA1} this implies that
\[ k\cE \cong k[X_1,\ldots,X_n]/(X_1^p,\ldots,X_n^p)\]
is a truncated polynomial ring. Hence the algebra $k\cE$ inherits a $\ZZ$-grading from the polynomial ring in $n$ variables such that $k\cE_{\ge 1} = \Rad(k\cE)$. We consider the linear 
projection $\pr_1 : k\cE_{\ge 1} \lra k\cE_1$. General theory cf.\ \cite[(6.1),(6.4)]{Ca} ensures that
\[ \Pt(\cE) \stackrel{\sim}{\lra} \pr^{-1}(k\cE_1\!\smallsetminus\!\{0\})=\Rad(k\cE)\!\smallsetminus\!\Rad^2(k\cE)\]  
is an open, conical subset of $k\cE_{\ge 1}$. Consequently,
\[ U_\cE:= \{ [x] \in \PP(k\cE_{\ge 1}) \ ;  \ \pr_1(x) \ne 0\}\]
is a dense open subset of the projective space $\PP(k\cE_{\ge 1})\cong \PP^{p^n-2}$. 

\bigskip

\begin{Thm} \label{EA3} Let $M$ be a $k\cE$-module of constant $j$-rank. Then the following statements hold:
\begin{enumerate}
\item The map 
\[ \msIm^j_M : U_\cE \lra \Gr_{\rk^j(M)}(M) \ \ ; \ \ u \mapsto \im(u^j_M)\]
is a homogeneous morphism.
\item If $\aurk(\cE) \ge 2$ and $V \subseteq \Rad(k\cE)$ is a subspace such that $\Rad(k\cE)=V\!\oplus\!\Rad^2(k\cE)$, then $\PP(V) \subseteq U_\cE$, and $\deg(\mspl_M\circ\msIm^j_M)=\deg(\mspl_M\circ\msIm^j_M|_{\PP(V)})$. \end{enumerate} \end{Thm}

\begin{proof} (1) Let $d:=\rk^j(M)$ and denote by $\varrho : k\cE \lra \End_k(M)$ the representation afforded by $M$. Then 
\[ \omega^j : U_\cE \lra \PP(\End_k(M))_d \ \ ; \ \ [u] \mapsto [\varrho(u)^j]\]
is a morphism and Proposition \ref{Pr2} shows that $\msIm^j_M$ also enjoys this property. Thanks to Lemma \ref{MPS3}, the morphism $\msIm^j_M$ is homogeneous.

(2) By assumption, we have $\dim_kV=\aurk(\cE)=:n$ and $k\cE\cong k[X_1,\ldots,X_n]/(X_1^p,\ldots,X_n^p)$. Setting $x_i:=X_i\!+\!(X_1^p,\ldots,X^p_n)$, we note that there exists an automorphism $\lambda \in \Aut(k\cE)$ such that 
$\lambda(\bigoplus_{i=1}^nkx_i)=V$. Observe that $\alpha \mapsto \lambda\circ\alpha$ defines a bijection $\Pt(\cE)\lra \Pt(\cE)$. Since $(\bigoplus_{i=1}^nkx_i)\!\smallsetminus\!\{0\} \subseteq \Pt(\cE)$,
we conclude that $V\!\smallsetminus\!\{0\} \subseteq \Pt(\cE)$. Consequently, $\PP(V)\subseteq U_\cE$. 

Let $\iota : \PP(V) \lra \PP(k\cE_{\ge 1})$ be the morphism of degree $1$ that is induced by the inclusion $V\subseteq k\cE_{\ge1}$. As $\im \iota \subseteq U_\cE$, Corollary \ref{MPS4} implies
\[ \deg(\mspl_M\circ\msIm^j_M) = \deg(\mspl_M\circ\msIm^j_M\circ\iota),\]
as desired. \end{proof}  

\bigskip

\begin{Definition} Let $\cE$ be an elementary abelian group scheme, $M$ be an $\cE$-module of constant $j$-rank. Then
\[ \deg^j(M):= \deg(\mspl_M\circ\msIm^j_M)\]
is called the {\it $j$-degree} of $M$.\end{Definition}

\bigskip
\noindent
Let $M$ be a module for a finite group scheme $\cG$. Given an automorphism $\lambda \in \Aut(k\cG)$ of the associative $k$-algebra $k\cG$, we consider the 
twisted module $M^{(\lambda)}$, which has underlying $k$-space $M$ and action
\[ a\dact m := \lambda^{-1}(a)m \ \ \ \ \ \ \ \forall \ a \in k\cG, \, m \in M.\]
For future reference we record the following direct consequence of the foregoing result, which shows that the $j$-degree of an $\cE$-module does not depend on the choice
of generators of $k\cE$.

\bigskip

\begin{Cor} \label{EA4} Let $M$ be an $\cE$-module of constant $j$-rank, $\lambda \in \Aut(k\cE)$ be an automorphism. Then $M^{(\lambda)}$ has constant $j$-rank and
\[ \deg^j(M^{(\lambda)}) = \deg^j(M).\] 
\end{Cor}

\begin{proof} Since $\alpha(t)^j_{M^{(\lambda)}} = (\lambda^{-1}\!\circ\!\alpha)(t)^j_M$ and $\alpha \mapsto \lambda^{-1}\circ \alpha$ is a bijection of $\Pt(\cE)$,
it readily follows that $M^{(\lambda)}$ has constant $j$-rank. For $\aurk(\cE)=1$, there is nothing to be shown. Alternatively, we pick a subspace $V \subseteq 
\Rad(k\cE)$ of dimension $\aurk(\cE)\ge 2$ such that $\Rad(k\cE)=V\!\oplus\!\Rad^2(k\cE)$. Since $\lambda$ is an automorphism, the subspace $W:=\lambda^{-1}(V)$ enjoys the same 
property. Moreover, $\lambda^{-1}$ induces a morphism $\lambda^{-1} : \PP(V) \lra \PP(W)$ of degree $1$. Now Theorem \ref{EA3} in conjunction with Corollary \ref{MPS5} yields
\begin{eqnarray*} 
\deg^j(M^{(\lambda)}) &=& \deg(\mspl_M\circ\msIm^j_{M^{(\lambda)}}|_{\PP(V)}) = \deg(\mspl_M\circ\msIm^j_M|_{\PP(W)}\circ \lambda^{-1}) = \deg(\mspl_M\!\circ\!\msIm^j_M|_{\PP(W)})\\ 
& = &\deg^j(M),
\end{eqnarray*}
as desired. \end{proof}

\bigskip 

\begin{Thm} \label{EA5} Let $\cE$ be an elementary abelian group scheme. If $M$ is an $\cE$-module of constant $j$-rank, then 
\[ \deg^j(M) + \deg^j(M^\ast) = j\rk^j(M).\]
\end{Thm}

\begin{proof} Setting $r:= \aurk(\cE)$, we observe that Lemma \ref{EA1} yields $k\cE = U_0(\fe_r)$, where $\Rad(k\cE)=\fe_r\oplus\Rad^2(k\cE)$. We denote by
$\eta_\cE$ and $\eta_{\fe_r}$ the antipodes of $k\cE$ and $U_0(\fe_r)$, respectively. Direct computation shows that
\[ M^\ast \cong (M^\sharp)^{(\lambda)},\]
where $M^\sharp$ is the dual of $M$ as a $U_0(\fe_r)$-module and $\lambda:=\eta_\cE\circ\eta_{\fe_r}$. 

By virtue of Theorem \ref{EA3}, we obtain
\[ \deg^j(M) = \deg(\mspl_M\circ \msIm^j_M|_{\PP(\fe_r)}) = \deg^j(\mspl_M\circ\msim^j_M),\]
implying that the $j$-degree of $M$ as a $k\cE$-module coincides with that of the $U_0(\fe_r)$-module $M$. A consecutive application of Theorem \ref{RDF2} and Corollary \ref{EA4} thus
yields
\[ j\rk^j(M) = \deg^j(M)+\deg^j(M^\sharp) = \deg^j(M)+\deg^j((M^\sharp)^{(\lambda)}) = \deg^j(M)+\deg^j(M^\ast),\]
as desired. \end{proof}

\bigskip
\noindent
Our final results of this section show that degrees of modules may be computed on elementary abelian groups of rank $2$.

\bigskip

\begin{Cor} \label{EA6} Let $\cE$ be an elementary abelian group scheme, $\cE'\subseteq\cE$ be a subgroup scheme such that $\aurk(\cE')\ge 2$. If $M$ is an $\cE$-module of constant
$j$-rank, then $\deg^j(M)=\deg^j(M|_{\cE'})$. \end{Cor}

\begin{proof} According to Corollary \ref{EA2}(2), the group $\cE'$ is elementary abelian. 

Now let $V \subseteq \Rad(k\cE')$ be a subspace such that $\Rad(k\cE')=V\!\oplus\!\Rad^2(k\cE')$. If $x \in V\!\smallsetminus\!\{0\}$, then $\alpha_x : \fA_p \lra k\cE'$ is a $p$-point (cf. \cite[\S 6]{Ca}). 
As $k\cE$ is a free $k\cE'$-module, the composite $\iota\circ\alpha_x$ of $\alpha_x$ and the natural inclusion $\iota : k\cE'\lra k\cE$ is a $p$-point. Thus, \cite[\S 6]{Ca} implies that $x \in \Rad(k\cE)\!
\smallsetminus\! \Rad^2(k\cE)$. Hence $V\cap \Rad^2(k\cE)=(0)$, and there exists a subspace $V\subseteq W\subseteq \Rad(k\cE)$ such that $\Rad(k\cE)=W\!\oplus\!\Rad^2(k\cE)$. Since $\dim_kV=
\aurk(\cE')\ge 2$, we may apply Theorem \ref{EA3} and Corollary \ref{MPS5} to arrive at
\[ \deg^j(M)= \deg(\mspl_M\circ\msIm^j_M|_{\PP(W)}) = \deg(\mspl_M\circ\msIm^j_M|_{\PP(V)})= \deg(\mspl_M\circ\msIm^j_{M|_{\cE'}}|_{\PP(V)})=\deg^j(M|_{\cE'}),\]
as desired. \end{proof}

\bigskip
\noindent
We denote by $Z(\cG)$ the {\it center} of the finite group scheme $\cG$, cf.\ \cite[(I.2.6)]{Ja}.

\bigskip

\begin{Cor} \label{EA7} Let $\cG$ be a finite group scheme, $M$ be a $\cG$-module of constant $j$-rank. 
\begin{enumerate}
\item If $\cE,\cE'$ are elementary abelian subgroups such that $\aurk(\cE\cap\cE')\ge 2$, then $\deg^j(M|_{\cE})=\deg^j(M|_{\cE'})$. 
\item If $\aurk(Z(\cG))\ge 2$, then there exists $d \in \NN_0$ such that $\deg^j(M|_\cE)=d$ for every elementary abelian subgroup $\cE\subseteq \cG$
such that $\aurk(\cE)\ge 2$. \end{enumerate}\end{Cor}

\begin{proof} (1) In view of Corollary \ref{EA6} and our current assumption, we have $\deg^j(M|_{\cE})=\deg^j(M|_{\cE\cap\cE'})=\deg^j(M|_{\cE'})$. 

(2) By assumption, there exists an elementary abelian subgroup $\cE_0 \subseteq Z(\cG)$ such that $\aurk(\cE_0)\ge 2$. Let $\cE\subseteq\cG$ be
elementary abelian of rank $\aurk(\cE)\ge 2$. Then $\cE\cE_0$ is elementary abelian, and a two-fold application of Corollary \ref{EA6} gives
\[ \deg^j(M|_{\cE}) = \deg^j(M|_{\cE\cE_0}) = \deg^j(M|_{\cE_0}),\]
so that the left-hand value does not depend on the choice of $\cE$. \end{proof}

\bigskip

\begin{Remark} Let $\cG$ be a finite group scheme and let $\fE(2\!\uparrow,\cG)$ be the set of elementary abelian subgroups of $\cG$ such that $\aurk(\cE)\ge 2$. We endow this set
with the structure of a graph by postulating that two elements $\cE,\cE'$ are linked by a bond, whenever $\aurk(\cE\cap\cE')\ge 2$. The above arguments show that the conclusion
of Corollary \ref{EA7}(2) holds whenever $\fE(2\!\uparrow,\cG)$ is connected. \end{Remark} 

\bigskip
\noindent
Since the category $\EIP(\cG)$ of equal images modules is closed under images and direct sums, every $\cG$-module $M$ possesses a unique largest submodule $\fK(M)$ belonging to $\EIP(\cG)$, the so-called {\it generic kernel} of $M$. 

\bigskip

\begin{Cor} \label{EA8} Let $\cE$ be an elementary abelian group scheme such that $\aurk(\cE)=2$. If $M$ is an $\cE$-module of constant rank, then
\[ \deg(M)=\dim_kM/\fK(M).\]
\end{Cor}

\begin{proof} We have $k\cE \cong U_0(\fe_2)$, where $\Rad(k\cE)=\fe_2\!\oplus\!\Rad^2(k\cE)$. It thus follows from Theorem \ref{EA3}(2) that $\deg(M)$ coincides with the degree
of $M$ as a $U_0(\fe_2)$-module. We consider the exact sequence
\[ (0) \lra \fK(M) \lra M \lra M/\fK(M) \lra (0).\]
As $M$ has constant rank, \cite[(7.6)]{CFS} implies $\sum_{x \in \fe_2\smallsetminus\{0\}}\ker x_M \subseteq \fK(M)$. The equal images module $\fK(M)$ has 
degree $\deg(\fK(M))=0$, and Lemma \ref{DGP4}(4) yields $\deg(M)=\deg(\fK(M))\!+\!\dim_kM/\fK(M)= \dim_kM/\fK(M)$. \end{proof}

\bigskip

\subsection{Self-dual modules}
Let $\alpha \in \Pt(\cG)$ be a $p$-point. If $M$ is a $\cG$-module, then the isomorphism class 
\[ \Jt(M,\alpha) = [\alpha^\ast(M)]\]
of the $\fA_p$-module $\alpha^\ast(M)$ is the {\it Jordan type of $M$ relative to $\alpha$}. We denote by
\[ \Jt(M) := \{\Jt(M,\alpha) \ ; \ \alpha \in \Pt(\cG)\}\]
the finite set of Jordan types of $M$ and say that $M$ has {\it constant Jordan type} if $|\Jt(M)|=1$. Since a $\cG$-module has constant Jordan type if and only if it has constant $j$-rank
for all $j \in \{1,\ldots,p\!-\!1\}$, Lemma \ref{FG1} ensures that our present definition is compatible with the earlier one.

\bigskip

\begin{Thm} \label{FG4} Let $\cG$ be a finite group scheme such that $\aurk(\cG)\ge 2$. If $M$ is a self-dual $\cG$-module, then the following statements hold:
\begin{enumerate}
\item If $M$ has constant $j$-rank, then $\rk^j(M)\equiv 0 \modd(2)$, whenever $j \equiv 1\modd(2)$.
\item If $M$ has constant Jordan type $\Jt(M)=\bigoplus_{i=1}^pa_i[i]$, then $a_i\equiv 0 \ \modd(2)$ whenever $i\equiv 0 \ \modd(2)$. \end{enumerate}\end{Thm}

\begin{proof} By assumption, there exists an elementary abelian subgroup of abelian unipotent rank $\ge 2$. We consider the self-dual $\cE$-module $N:=M|_\cE$.

Let $j \in \{1,\ldots,p\!-\!1\}$ be odd. If $M$ has constant $j$-rank, then $N$ has constant $j$-rank $\rk^j(N)=\rk^j(M)$, and assertion (1) follows from Theorem \ref{EA5}. The second statement
is a direct consequence of (1) and the formula $a_j = \rk^{j-1}(N)\!-\!2\rk^j(N)\!+\!\rk^{j+1}(N)$.\end{proof}

\bigskip

\section{The equal images property} \label{S:EIP}
We provide further applications of our morphisms $\msim^j_M$ by establishing conditions for a module of constant rank to have the equal images property. Our starting point is the following observation
concerning restricted Lie algebras:

\bigskip

\begin{Proposition} \label{EIP1} Let $\fg$ be a $p$-trivial restricted Lie algebra, $M$ be a $U_0(\fg)$-module of constant rank. If there exist linearly independent elements $x,y \in \fg\!\smallsetminus\!\{0\}
$ such that $\im x_M = \im y_M$, then $M$ has the equal images property. \end{Proposition}

\begin{proof} By assumption, the map $\msim_M : \PP(\fg) \lra \Gr_{\rk(M)}(M)$ is not injective. The assertion thus follows from Corollary \ref{LASN3}(2). \end{proof}

\bigskip

\begin{Remark} Suppose that $r\ge 2$. If $M$ is a $U_0(\fe_r)$-module of constant rank such that $M|_{\fe_2}$ has the equal images property, then Proposition \ref{EIP1} implies that $M$ has the 
equal images property. In view of Corollary \ref{EA6}, an analogous result holds for elementary abelian group schemes. \end{Remark} 

\bigskip
\noindent
Recall that $\GG_{a(r)}$ denotes the $r$-th Frobenius kernel of the additive group $\GG_a = {\rm Spec}(k[T])$. Thus, $k[\GG_{a(r)}] = k[T]/(T^{p^r})$ is generated by the primitive element $t:= T\!+\!(T^{p^r})$. If $u_0,\ldots, u_{r-1} \in k\GG_{a(r)}$ are given by $u_i(t^j):=\delta_{p^i,j}$, then the map $X_i \mapsto u_i$ that is defined on the indeterminates $X_0,\ldots,X_{r-1}$ over $k$ induces an isomorphism
\[ k[X_0,\ldots,X_{r-1}]/(X_0^p,\ldots,X_{r-1}^p) \stackrel{\sim}{\lra} k\GG_{a(r)}.\]
Given a finite group scheme $\cG$ and $r \in \NN$, we denote by $\cG_r$ the {\it $r$-th Frobenius kernel} of $\cG$. By definition, $\cG_r$ is the scheme-theoretic kernel of the $r$-th iterate
$F^r : \cG \lra \cG^{(r)}$ of the Frobenius homomorphism $F : \cG \lra \cG^{(1)}$. For $\cG=\GG_{a(r)}$, $F$ may be viewed as an endomorphism. Moreover, we have $\ker F^s = \GG_{a(s)}$
and $F^s$ induces an isomorphism $\GG_{a(r)}/\GG_{a(s)} \cong \GG_{a(r-s)}$. Note that the induced homomorphism $F^s : k\GG_{a(r)} \lra k\GG_{a(r)}$ sends the canonical generator $u_i$ 
onto $u_{i-s}$. 

\bigskip

\begin{Theorem} \label{EIP2} Let $\cG$ be an infinitesimal group. If $M$ is a $\cG$-module of constant rank such that $M|_{\cG_2} \in \EIP(\cG_2)$, then $M$ has the equal images property. 
\end{Theorem}

\begin{proof} Suppose that $\cG$ has height $r$. For $r\le 2$, there is nothing to be shown, so we assume that $r\ge 3$.

We first consider the case, where $\cG=\GG_{a(r)}$. The map $F^2 : \GG_{a(r)} \lra \GG_{a(r)}$ is given by $x \mapsto x^{p^2}$. This implies that $k[(\GG_{a(r)})_2] = k[\GG_{a(r)}]/(t^{p^2})$ as well
as $k(\GG_{a(r)})_2 = k[u_0,u_1]$. Setting $\fe_r := \bigoplus_{i=0}^{r-1}ku_i$, we have $k\GG_{a(r)}=U_0(\fe_r)$ as well as $k(\GG_{a(r)})_2 = U_0(\fe_2)$. Our assertion thus follows by applying
Proposition \ref{EIP1} to the elements $u_0$ and $u_1$.

In the general case, we consider the inclusion $(\GG_{a(r)})_2 = \GG_{a(2)} \subseteq \GG_{a(r)}$, which corresponds to the inclusion $k\GG_{a(2)} = k[u_0,u_1] \subseteq k[u_0,\ldots, u_{r-1}] = 
k\GG_{a(r)}$. By assumption, there exist subspaces $V_1,\ldots, V_{p-1} \subseteq M$ such that
\[ (\ast) \ \ \ \ \ \ \ \ \ \im \psi(u_1)^j_M = V_j \ \ \ \ \ \ \ \ \ \forall \ \psi \in V(\cG_2)\smallsetminus\!\{\varepsilon\}, \, j \in \{1,\ldots,p\!-\!1\}.\]
We fix $j \in \{1,\ldots, p\!-\!1\}$ and let $\varphi \in V(\cG)\!\smallsetminus\!\{\varepsilon\}$. Then there exists $s \le r\!-\!1$ such that $\ker \varphi = \GG_{a(s)}$ and $\GG_{a(r)}/\ker\varphi \cong 
\GG_{a(r-s)}$. Our observations above now provide an injective homomorphism $\zeta : \GG_{a(r-s)} \lra \cG$ such that
\[ \varphi = \zeta \circ F^s.\]
We consider the $\GG_{a(r-s)}$-module $N:= \zeta^\ast(M)$. It follows that $N$ has constant rank $\rk(M)$.

If $s = r\!-\!1$, then $\zeta : \GG_{a(1)} \lra \cG$ satisfies
\[\varphi = \zeta \circ F^{r-1} = (\zeta \circ F)\circ F^{r-2}.\]
Accordingly, the map $\zeta\circ F : \GG_{a(2)} \lra \cG$ factors through $\cG_2$ and $\zeta \circ F \in V(\cG_2)\!\smallsetminus\!\{\varepsilon\}$. Consequently, ($\ast$) yields
\[ \im\varphi(u_{r-1})^j_M = \im (\zeta \circ F)(u_1)^j_M = V_j,\]
as desired. 

Alternatively, $s\le r\!-\!2$, so that $\GG_{a(2)}\subseteq \GG_{a(r-s)}$. If $\lambda : \GG_{a(2)} \lra \GG_{a(2)}$ is a non-trivial homomorphism, then $\zeta \circ \lambda \in V(\cG_2)\!\smallsetminus \!
\{\varepsilon\}$, and 
\[ \im \lambda(u_1)^j_{N} = \im (\zeta\circ \lambda)(u_1)^j_M = V_j.\]
Consequently, $N|_{\GG_{a(2)}} \in \EIP(\GG_{a(2)})$, and the first part of the proof ensures that $N \in \EIP(\GG_{a(r-s)})$. In particular, $\im \zeta(u_{r-s-1})^j_M = V_j$, whence 
$\im \varphi(u_{r-1})^j_M = \im (\zeta\circ F^s)(u_{r-1})^j_M = \im \zeta(u_{r-s-1})^j_M = V_j$. In view of Lemma \ref{FG1}, the module $M$ thus has the equal images property. \end{proof}

\bigskip 

\begin{Remarks} (1) Since the restriction $M|_{\GG_{a(1)}}$ of an arbitrary $\GG_{a(r)}$-module $M$ has the equal images property, the foregoing result may fail if $\cG_2$ is replaced by $\cG_1$.

(2) Since $\aurk(\GG_{a(r)})=r$, it follows from Corollary \ref{EA6} that $\deg^j(M)=\deg^j(M|_{\GG_{a(2)}})$ for every $\GG_{a(r)}$-module of constant $j$-rank. \end{Remarks}

\bigskip

\begin{Lemma} \label{EIP3} Let $\cU$ be an abelian unipotent group scheme, $M$ be a $\cU$-module. 
\begin{enumerate}
\item If $M|_{\cE_{\cU}}$ has constant $j$-rank, then $M$ has constant $j$-rank.
\item If $M|_{\cE_{\cU}} \in \EIP(\cE_{\cU})$, then $M \in \EIP(\cU)$. \end{enumerate} \end{Lemma}

\begin{proof} Thanks to Lemma \ref{EA1}(2), the canonical inclusion $\iota : \cE_{\cU} \lra \cU$ induces a bijection $\iota_\ast : \Pp(\cE_{\cU}) \lra \Pp(\cU)$. 

(1) Given $j \in \{1,\ldots,p\!-\!1\}$, we put 
\[ \Pp^j(\cU)_M:=\{x \in \Pp(\cU) \ ; \ \text{there is} \ \alpha \in x \ \text{such that} \ \rk(\alpha(t)^j_M)<\rk^j(M)\}.\]
Thanks to \cite[(4.5)]{FPe10}, $\Pp^j(\cU)_M$ is a closed subset of $\Pp(\cU)$ such that $M$ has constant $j$-rank if and only if $\Pp^j(\cU)_M=\emptyset$. The observation above in conjunction with
\cite[(3.4)]{FPe10} yields
\[ \Pp^j(\cU)_M = \Pp^j(\cE_{\cU})_{M|_{\cE_{\cU}}}.\]
By assumption, the latter set is empty, so that $M$ has constant $j$-rank.

(2) By assumption, there exists a vector space $V\subseteq M$ such that 
\[ \im \alpha(t)_M = V \ \ \ \ \text{for all} \ \alpha \in \Pt(\cE_{\cU}).\]
In particular, the module $M|_{\cE_{\cU}}$ has constant rank. By (1), this implies that $M$ is a $\cU$-module of constant rank. 

Using the notation of the proof of Lemma \ref{FG1}(2), we observe that the identity $\Pp(\cU)=\iota_\ast(\Pp(\cE_\cU))$, provides $p$-points $\alpha_1, \ldots, \alpha_s \in \Pt(\cE_\cU)$ such that
\[ (v_1^{p^{n_1-1}},\ldots, v_s^{p^{n_s-1}}) = (\alpha_1(t),\ldots, \alpha_s(t))\]
is the ideal of all elements of $u \in k\cU$ with $u^p=0$. It follows that 
\[ \im (v_j^{p^{n_j-1}})_M \subseteq V \ \ \ \ \ \text{for all} \ j \in \{1,\ldots,s\}.\]
Now let $\alpha \in \Pt(\cU)$ be a $p$-point. Then we have $\alpha(t) \in (v_1^{p^{n_1-1}},\ldots, v_s^{p^{n_s-1}})$, so that
$\im \alpha(t)_M \subseteq V$. As $M$ has constant rank, these spaces are in fact equal. Since $\cU$ is abelian, this implies that $M$ has the equal images property.  \end{proof}

\bigskip

\begin{Examples} Suppose that $U$ is an abelian $p$-group, $U(p):=\{u\in U \ ; \ u^p=1\}$.
\begin{enumerate}
\item Let $N \in \EIP(U(p))$ and consider $M := kU\!\otimes_{kU(p)}\!N$. Since $M|_{U(p)}\cong N^{[U:U(p)]}$, it follows from Lemma \ref{EIP3} that $M$ is an equal 
images module.
\item Contrary to $p$-elementary abelian groups, modules of constant rank $0$ may not be trivial. The $U$-module $M:=kU\!\otimes_{kU(p)}\!k$ is a trivial 
$U(p)$-module, so that Lemma \ref{EA1} implies that $M \in \EIP(U)$ has constant rank $0$. However, if $U(p) \subsetneq U$, then $M$ is not a trivial $U$-module.\end{enumerate} \end{Examples} 

\bigskip

\begin{Theorem} \label{EIP5} Let $\cG$ be a finite group scheme such that $\aurk(Z(\cG))\ge 2$. If $M$ is a $\cG$-module of constant rank such that $M|_{Z(\cG)} \in \EIP(Z(\cG))$, then 
$M \in \EIP(\cG)$. \end{Theorem}

\begin{proof} Since $\aurk(Z(\cG)) \ge 2$, Lemma \ref{FG4} provides an elementary subgroup $\cE_0 \subseteq Z(\cG)$ such that $\aurk(\cE_0)=2$. We let $\alpha_{\cE_0} \in \Pt(\cE_0)$ be a 
$p$-point.

Let $\cU \subseteq \cG$ be a maximal abelian unipotent subgroup of $\cG$. Being an image of the local algebra $k(\cE_0\!\times\!\cU)\cong k\cE_0\!\otimes_k\!k\cU$,  the algebra $k(\cE_0\cU)$ is 
local, so that the group $\cE_0\cU$ is unipotent. As $\cE_0$ belongs to the center of $\cG$, the group $\cE_0\cU$ is also abelian. We thus have $\cE_0\cU=\cU$, whence $\cE_0\subseteq \cU$ and
$\cE_0 \subseteq \cE_\cU$. 

By our current assumption, the $\cE_\cU$-module $N:=M|_{\cE_\cU}$ has constant rank, so that Corollary \ref{EA6} yields $\deg(N)=\deg(N|_{\cE_0})$. Since $N|_{\cE_0}\in \EIP(\cE_0)$, we obtain
$\deg(N)=0$. As $\cE_\cU$ is abelian, it follows that $N \in \EIP(\cE_\cU)$, while Lemma \ref{EIP3} implies $M|_{\cU} \in \EIP(\cU)$. 

Now let $\alpha \in \Pt(\cG)$ be a $p$-point. Then there exists a maximal abelian unipotent group $\cU$ containing $\im \alpha$. By the above, we have
\[ \im \alpha(t)^j_M = \im \alpha_{\cE_0}(t)^j_M \ \ \ \ \text{for all} \ j \in \{1,\ldots,p\!-\!1\}.\]
As a result, the $\cG$-module $M$ has the equal images property. \end{proof}

\bigskip

\begin{Remark} Passage to dual modules provides analogous results for equal kernels modules. We leave the details to the interested reader.\end{Remark}

\end{document}